\documentclass[11 pt, oneside]{amsart}
\usepackage[top=1.5in,bottom=1in,left=1in,right=1in]{geometry}
\usepackage{amsmath,amssymb,amsthm,amsopn,extarrows,accents,faktor,enumitem,stmaryrd}  
\usepackage{colonequals}

\usepackage[linktocpage=true]{hyperref}
\usepackage[dvipsnames]{xcolor}

\hypersetup{
    colorlinks = true,
    allcolors = {blue}}
\usepackage[dvipsnames]{xcolor}

\usepackage{setspace}
\usepackage{tikz-cd,tikz}
\usetikzlibrary{matrix,arrows,decorations.pathmorphing}

\newenvironment{thmenum}
 {\begin{enumerate}[label=\upshape(\arabic*),ref=\theexample(\arabic*),leftmargin=20pt]}
 {\end{enumerate}}

\newtheorem{lemma}[equation]{Lemma}
\newtheorem{theorem}[equation]{Theorem}
\newtheorem{proposition}[equation]{Proposition}
\newtheorem{corollary}[equation]{Corollary}
\newtheorem*{theoremp}{Theorem}
\newtheorem*{propositionp}{Proposition}
\theoremstyle{definition}
\newtheorem{definition}[equation]{Definition}
\newtheorem*{remark*}{Remark}
\newtheorem{remark}[equation]{Remark}
\newtheorem{example}[equation]{Example}

\numberwithin{equation}{section}

\newcommand{\hooklongrightarrow}{\lhook\joinrel\longrightarrow}

\newcommand{\Ts}{ T}

\newcommand{\Gbar}{\overline{G_{\ad}}}
\newcommand{\Zbar}{\overline{\mathfrak{Z}}}
\newcommand{\Z}{\mathfrak{Z}}
\newcommand{\Tbar}{\overline{T}}
\newcommand{\Tlog}{T^\ast_D\Gbar}
\newcommand{\Vino}{\accentset{\circ}{V}_G}
\newcommand{\Genh}{G_{\text{enh}}}
\newcommand{\Stab}{\text{Stab}}
\newcommand{\mubar}{\overline{\mu}}
\newcommand{\alphao}{\accentset{\circ}{\alpha}}
\newcommand{\tauo}{\accentset{\circ}{\tau}}
\newcommand{\immub}{\Delta}
\newcommand{\pib}{\overline{\pi}}

\newcommand{\DG}{\mathbf{D}_{G_{\ad}}}
\newcommand{\gr}{\mathbf{Gr}_{G^\vee}}
\newcommand{\DGlog}{\mathbf{D}_{\overline{G_{\ad}}}}
\newcommand{\DGlogi}{\mathbf{D}_{\overline{G_{\ad}},I}}

\DeclareMathOperator{\im}{im}
\DeclareMathOperator{\spec}{Spec}
\DeclareMathOperator{\Ad}{Ad}
\DeclareMathOperator{\ad}{ad}

\title{Steinberg slices and group-valued moment maps}
\author{Ana B\u{a}libanu}
\address{Department of Mathematics, Harvard University, 1 Oxford Street, Cambridge, MA  02138, USA}
\email{ana@math.harvard.edu}
\date{}

\begin{document}
\maketitle

\begin{abstract}
We define a class of transversal slices in spaces which are quasi-Poisson for the action of a complex semisimple group $G$. This is a multiplicative analogue of Whittaker reduction. One example is the multiplicative universal centralizer $\Z$ of $G$, which is equipped with the usual symplectic structure in this way. We construct a smooth relative compactification $\Zbar$ by taking the closure of each centralizer fiber in the wonderful compactification of $G$. By realizing this relative compactification as a transversal in a larger quasi-Poisson variety, we show that it is smooth and log-symplectic.
\end{abstract}

\setcounter{tocdepth}{1}
\tableofcontents

\section*{Introduction}
Let $G$ be a simply-connected, complex semisimple group, and let $G_{\ad}$ be its adjoint form. The group $G_{\ad}$ acts on $G$ by conjugation, and $G$ contains a transversal slice $\Sigma$ for this action which was introduced by Steinberg \cite{ste:65}. The resulting (multiplicative) universal centralizer is the smooth affine variety
\[\Z\colonequals \left\{(a,h)\in G_{\ad}\times\Sigma\mid a\in G_{\ad}^h\right\}.\]
This family of centralizers first appeared in work of Lusztig \cite[Section 8.6]{lus:76}. When $G$ is simply-laced, Bezrukavnikov, Finkelberg, and Mirkovic \cite{bez.fin.mir:05} showed that its coordinate ring is isomorphic to the equivariant $K$-theory of the affine Grassmannian of the Langlands dual group $G^\vee$---therefore, in this case $\Z$ is an example of a Coulomb branch as defined by Nakajima \cite{nak:17}. 

The natural symplectic structure on $\Z$ is inherited from the nondegenerate quasi-Poisson structure on the double $\DG\colonequals G_{\ad}\times G$ as described, up to a finite central quotient, by Finkelberg and Tsymbaliuk \cite{fin.tsy:19}. We construct a relative compactification 
\[\Zbar\colonequals \left\{(a,h)\in \Gbar\times\Sigma\mid a\in \overline{G_{\ad}^h}\right\}\]
of $\Z$, by taking the closure of each centralizer fiber inside the wonderful compactification $\Gbar$. We show that $\Zbar$ is smooth, and that the symplectic structure on $\Z$ extends to a log-symplectic Poisson structure on $\Zbar$. When $G$ is simply-laced, the relative compactification $\Zbar$ can be obtained from the Coulomb branch definition given in \cite{bez.fin.mir:05} through a Rees construction which comes from the Vinberg monoid. Therefore, in this case $\Zbar$ is an example of a Rees algebra approach to compactified Couloumb branches suggested by Braverman, Finkelberg, and Nakajima \cite[Remark 3.7]{bra.fin.nak:19}.

\subsection*{The additive setting}
Let $\mathfrak{g}$ be the Lie algebra of $G$ and fix a regular $\mathfrak{sl}_2$-triple $\{e,h,f\}$. The regular nilpotent element $e$ is contained in a unique Borel subalgebra $\mathfrak{b}$, and we write $\mathfrak{u}$ for its nilpotent radical and $U$ for the corresponding subgroup of $G$. 

Suppose that $M$ is a Poisson manifold with a Hamiltonian action of $G$ and, identifying the Lie algebra $\mathfrak{g}$ with its dual via the Killing form isomorphism, let $\nu\colon M\longrightarrow\mathfrak{g}$ be the associated moment map. The \emph{Whittaker reduction} of $M$ is the Poisson manifold
\[\nu^{-1}(f+\mathfrak{b})/U\]
obtained by Hamiltonian reduction with respect to the action of $U$ at a point corresponding to a regular character in $\mathfrak{u}^*$.

On the other hand, one can consider the principal slice of regular elements 
\[\mathcal{S}\colonequals f+\mathfrak{g}^e\subset\mathfrak{g},\] 
which was defined by Kostant \cite{kos:63} and which is contained in the space $f+\mathfrak{b}$. Because $\mathcal{S}$ meets every regular adjoint orbit exactly once and transversally, it is a Poisson transversal for the Kirillov--Kostant--Souriau Poisson structure on $\mathfrak{g}$. This implies that the preimage $\nu^{-1}(\mathcal{S})$ is a Poisson transversal in $M$, and in this way inherits a natural Poisson structure.

These two constructions are related through the Kostant Cross-section Theorem \cite[Theorem 8]{kos:63}, which states that $\mathcal{S}$ is a slice for the free action of $U$ on $f+\mathfrak{b}$. The induced isomorphism of Poisson manifolds
\[\nu^{-1}(\mathcal{S})\cong \nu^{-1}(f+\mathfrak{b})/U,\]
shows that classical Whittaker reduction coincides with the preimage of the Kostant section $\mathcal{S}$ under the moment map.

This approach is used in \cite{bal:19} to construct a canonical relative compactification of the universal centralizer 
\[\mathcal{Z}\colonequals \left\{(a,x)\in G_{\ad}\times\mathcal{S}\mid a\in G_{\ad}^x\right\}.\]
This is a symplectic variety obtained by Whittaker reduction relative to the $G\times G$-action on the cotangent bundle $T^*G_{\ad}$. The canonical symplectic structure on $T^*G_{\ad}$ extends to a log-symplectic structure on the logarithmic cotangent bundle $\Tlog$. It is shown in \cite{bal:19} that the universal centralizer $\mathcal{Z}$ has a smooth, log-symplectic relative compactification 
\[\overline{\mathcal{Z}}\colonequals \left\{(a,x)\in \Gbar\times\mathcal{S}\mid a\in\overline{G_{\ad}^x}\right\},\]
which is the Whittaker reduction of $\Tlog$. In view of the discussion above, there is a commutative diagram of moment maps
\begin{equation*}
\begin{tikzcd}[row sep=large]
T^*G_{\ad}\arrow[r, hook] \arrow[rd, swap, "\nu"]&\Tlog\arrow[d, "\overline{\nu}"]\\
					&\mathfrak{g}\times\mathfrak{g}.
\end{tikzcd}
\end{equation*}
The varieties $\mathcal{Z}$ and $\overline{\mathcal{Z}}$ are simply the preimages of the principal slice $\mathcal{S}\times(-\mathcal{S})$ under $\nu$ and $\overline{\nu}$, and their Poisson structures are therefore also obtained via restriction in this way.

\subsection*{Summary of results}
We give a multiplicative analogue of these results by considering manifolds which are quasi-Poisson relative to the action of $G$. These can be viewed as deformations of ordinary Poisson structures in which the Jacobi identity is twisted by a canonical trivector field induced by the group action. They were introduced in a series of papers by Alekseev, Malkin, and Meinrenken \cite{ale.mal.mei:98}, Alekseev and Kosmann-Schwarzbach \cite{ale.kos:00}, and Alekseev, Kosmann-Schwarzbach, and Meinreken \cite{ale.kos.mei:02}. These manifolds come equipped with group-valued momentum maps, and they are foliated by nondegenerate leaves. 

The geometry of quasi-Poisson structures makes the multiplicative setting more subtle than the additive case described above. In this setting there is no na\"ive analogue of Whittaker reduction, because one cannot generally perform quasi-Hamiltonian reduction with respect to the action of a subgroup. This is because a quasi-Poisson $G$-manifold is generally not quasi-Poisson for the action of a subgroup of $G$. To fix this issue, we introduce a multiplicative counterpart of Whittaker reduction which uses transversal slices. Kostant's principal slice $\mathcal{S}$ is replaced by the Steinberg cross-section $\Sigma$ of $G$, and we prove the following theorem, as Corollary \ref{slice-cor}:

\begin{theoremp}
Let $M$ be a quasi-Poisson $G$-manifold with group-valued moment map $\Phi\colon M\longrightarrow G$. The Steinberg slice
\[M_\Sigma\colonequals \Phi^{-1}(\Sigma)\]
is a smooth submanifold of $M$. The quasi-Poisson structure on $M$ pulls back to a Poisson structure on $M_\Sigma$ whose symplectic leaves are the intersections of $M_\Sigma$ with the nondegenerate leaves of $M$.
\end{theoremp}

The natural Poisson structure on $M_\Sigma$ is ``transverse'' to the quasi-Poisson structure on $M$, in the sense that it intersects every nondegenerate leaf transversally and symplectically. We use this approach to construct multiplicative analogues of several Whittaker-type algebraic varieties. We also define the notion of \emph{log-nondegenerate} quasi-Poisson structures, which are a multiplicative analogue of log-symplectic structures. In Proposition \ref{prop:stein-lognon} we show that the log-nondegeneracy condition behaves well with respect to Steinberg slices:

\begin{propositionp}
If $M$ is a log-nondegenerate quasi-Poisson $G$-manifold, the Poisson structure induced on the Steinberg slice $M_\Sigma$ is log-symplectic.
\end{propositionp}

We then apply this framework to the multiplicative universal centralizer $\mathfrak{Z}$. This variety sits as a Steinberg slice in the double $\DG=G_{\ad}\times G$, which is the quasi-Poisson analogue of the cotangent bundle of $G$. In fact, using the identification $T^*G_{\ad}\cong G_{\ad}\times\mathfrak{g}$, the cotangent bundle $T^*G_{\ad}$ is a bundle of Lie algebras over $G_{\ad}$ and $\DG$ is the simply-connected group scheme which integrates it. In Proposition \ref{logdouble} we extend the group scheme  $\DG$ to a group scheme over $\Gbar$, which we call the \emph{logarithmic double}.

\begin{propositionp}
The double $\DG$ extends to a group scheme $\DGlog$ over $\Gbar$ which integrates the bundle of Lie algebras $\Tlog$. The nondegenerate quasi-Poisson structure on $\DG$ extends to a log-nondegenerate quasi-Poisson structure on $\DGlog$. 
\end{propositionp}

This produces a commutative diagram of group-valued moment maps
\begin{equation*}
\begin{tikzcd}[row sep=large]
\DG	\arrow[r, hook] \arrow[rd, swap, "\mu"]&\DGlog\arrow[d, "\mubar"]\\
					&G\times G.
\end{tikzcd}
\end{equation*}
We write $\iota\colon G\longrightarrow G$ for the inversion map, and we prove the following theorem, as Theorem \ref{bigmain}:

\begin{theoremp}
There is an isomorphism of varieties
\[\Zbar\cong\mubar^{-1}(\Sigma\times\iota(\Sigma))\]
between the compactified universal centralizer $\Zbar$ and the Steinberg slice $\mubar^{-1}(\Sigma\times\iota(\Sigma))$ of the logarithmic double $\DGlog$. Therefore $\Zbar$ is a smooth, log-symplectic Poisson variety whose open dense symplectic leaf is $\Z$.
\end{theoremp}

\subsection*{Outline} 
In Section \ref{sec:background} we review quasi-Poisson manifolds as developed in \cite{ale.kos.mei:02}. We also outline how they fit into the framework of twisted Dirac structures, as described in \cite{bur.cra:05} and \cite{bur.cra:09}. In Section \ref{sec:stein} we develop a theory of Steinberg slices in quasi-Poisson manifolds. We give several examples of these slices, including a multiplicative analogue of the twisted cotangent bundle of the base affine space. Then we define the notion of log-nondegeneracy for quasi-Poisson manifolds, and we show that Steinberg slices in log-nondegenerate quasi-Poisson manifolds are log-symplectic.

In Section \ref{sec:mid} we recall the multiplicative universal centralizer $\Z$, which is a Steinberg slice in the double $\DG$, and we review some facts about the geometry of the wonderful compactification $\Gbar$. In Section \ref{sec:dlog} we use the Vinberg monoid to construct the smooth group scheme $\DGlog$. Then we show that the quasi-Poisson structure on $\DG$ extends to a log-nondegenerate structure on $\DGlog$. Finally, in Section \ref{sec:part} we realize the relative compactification $\Zbar$ as a Steinberg slice in $\DGlog$, equipping it with a log-symplectic Poisson structure. We give an explicit description of its stratification by symplectic leaves.

\subsection*{Acknowledgements} 
The author would like to thank Maxence Mayrand and Eckhard Meinrenken for interesting discussions, and the anonymous referee for detailed comments which have substantially improved this paper. This work was partially supported by a National Science Foundation MSPRF under award DMS--1902921.\\

%
%
%
%
%
%
%
\section{Quasi-Hamiltonian and quasi-Poisson structures}
\label{sec:background}
We recall the basics of quasi-Hamiltonian and quasi-Poisson manifolds below, and we refer to \cite{ale.kos.mei:02} for more details. We then explain how to view quasi-Poisson manifolds as twisted Dirac manifolds, following \cite{bur.cra:05} and \cite{bur.cra:09}. We will use this formalism in Section \ref{sec:stein}.

\subsection{Quasi-Poisson manifolds}
Let $G$ be a simply-connected, semisimple complex group, let $\mathfrak{g}$ be its Lie algebra, and write $(\cdot,\cdot)$ for the Killing form. Under the isomorphism $\mathfrak{g}\cong\mathfrak{g}^*$ induced by this form, the \emph{Cartan $3$-tensor} $\varphi\in\wedge^3\mathfrak{g}$ is the dual of the invariant trilinear function $\eta\in\wedge^3\mathfrak{g}^*$ given by
\[\eta(x,y,z)=\frac{1}{12}(x,[y,z])\qquad\text{for all }x,y,z\in\mathfrak{g}.\]
Let $\{e_i\}$ be a basis of $\mathfrak{g}$ which is orthonormal relative to the Killing form. Then
\[\varphi=\frac{1}{12}C_{ijk}e_i\wedge e_j\wedge e_k\]
where $C_{ijk}=(e_i,[e_j,e_k])$ are the structure constants. Here and throughout the paper we adopt the convention of summing over repeated indices.

If $G$ acts on a complex manifold $M$, we write $\xi_M$ for the polyvector field induced by the infinitesimal action of an element $\xi\in\wedge^k\mathfrak{g}$. In particular, the Cartan $3$-tensor $\varphi$ generates a trivector field $\varphi_M\in\Gamma(\wedge^3TM)$. A \emph{quasi-Poisson} structure on the manifold $M$ is a $G$-invariant section $\pi\in \Gamma(\wedge^2TM)$ such that 
\begin{equation}
\label{qpcond}
[\pi,\pi]=\varphi_M,
\end{equation}
where the bracket on the left is the Schouten--Nijenhuis bracket. In the special case where $G$ is abelian, the Cartan $3$-tensor is trivial, and a quasi-Poisson structure on $M$ is simply a $G$-invariant Poisson structure.

\begin{example}{\cite[Section 3]{ale.kos.mei:02}}
\label{grp}
The group $G$, equipped with the conjugation action, has a natural quasi-Poisson bivector 
\[\pi_{G}\colonequals \frac{1}{2}e_i^R\wedge e_i^L.\]
Here $e_i^L$ and $e_i^R$ are the invariant vector fields on $G$ corresponding to left- and right-multiplication. The bivector $\pi_G$ is tangent to the conjugacy classes, and it induces a quasi-Poisson structure on each one.
\end{example}

If $(M_1,\pi_1)$ and $(M_2,\pi_2)$ are quasi-Poisson $G$-manifolds, a $G$-equivariant map $f\colon M_1\longrightarrow M_2$ is called \emph{quasi-Poisson} if the bivectors $\pi_1$ and $\pi_2$ are $f$-related. A quasi-Poisson manifold $(M,\pi)$ is \emph{Hamiltonian} if it has a $G$-equivariant group-valued moment map
\[\Phi\colon M\longrightarrow G\]
which satisfies a differential equation analogous to the usual moment map condition \cite[Definition 2.2]{ale.kos.mei:02}. In particular, $\Phi$ is a quasi-Poisson map when $G$ is equipped with the bivector $\pi_G$. In what follows all quasi-Poisson manifolds will be Hamiltonian, so we will suppress this adjective.

\begin{example}
\label{double}
\cite[Example 5.3]{ale.kos.mei:02} Consider the \emph{internal fusion double} $D(G)\colonequals G\times G$. The group $G\times G$ acts on $D(G)$ by 
\[(g,h)\cdot (u,v)=\left(guh^{-1},hvg^{-1}\right)\]
for $(g,h)\in G\times G$ and $(u,v)\in D(G)$. Let $\{e_i^1,e_i^2\}$ be the induced orthonormal basis for the Lie algebra $\mathfrak{g}\oplus\mathfrak{g}$. The manifold $D(G)$ has a quasi-Poisson structure
\[\frac{1}{2}\left(e_i^{1L}\wedge e_i^{2R}+e_i^{1R}\wedge e_i^{2L}\right).\]
The associated moment map is
\begin{align*}
&D(G)\longrightarrow G\times G \\
	&(u,v)\longmapsto (uv, u^{-1}v^{-1}).
	\end{align*}
\end{example}

In the subsequent sections we will often use the reparametrization of $D(G)$ given by setting $a=u$ and $b=vu$. This is analogous to the left-trivialization of the cotangent bundle $T^*G$. In these coordinates the $G\times G$-action is 
\begin{equation}
\label{paramact}
(g,h)\cdot (a,b)=\left(gah^{-1},hbh^{-1}\right).
\end{equation}
At the point $(a,b)$ the quasi-Poisson bivector becomes
\[\frac{1}{2}\left(e_i^{1L}\wedge e_i^{2R}+e_i^{2L}\wedge e_i^{2R}+e_i^{1R}\wedge (\Ad_{a^{-1}}e_i)^{2L}\right).\]
Using the fact that $\Ad_{a^{-1}}$ is an orthogonal transformation relative to the Killing form and summing once again over repeated indices, the last term simplifies to
\[e_i^{1R}\wedge (\Ad_{a^{-1}}e_i)^{2L}=(\Ad_{a^{-1}}e_i)^{1L}\wedge (\Ad_{a^{-1}}e_i)^{2L}=e_i^{1L}\wedge e_i^{2L}.\]
Therefore the quasi-Poisson structure in these coordinates is
\begin{align}
\label{firstbiv}
\pi\colonequals \frac{1}{2}\left(e_i^{1L}\wedge \left(e_i^{2L}+e_i^{2R}\right)+e_i^{2L}\wedge e_i^{2R}\right),
\end{align}
and the associated moment map is
\begin{align}
\label{moment}
\mu\colon &D(G)\longrightarrow G\times G \\
	&(a,b)\longmapsto (aba^{-1}, b^{-1}).\nonumber
	\end{align}

Quasi-Poisson structures are not compatible with restriction to the action of a subgroup---that is, a quasi-Poisson $G$-manifold is not in general quasi-Poisson for the action of a subgroup of $G$. An exception to this is the case of diagonal subgroups, for which there is a procedure called \emph{internal fusion} \cite[Section 5]{ale.kos.mei:02} which we now describe. 

Suppose that $(M,\pi)$ is a quasi-Poisson $G\times G\times H$-manifold with group-valued moment map
\begin{align*}
M&\longrightarrow G\times G\times H \\
	m&\longmapsto (\Phi_1(m),\Phi_2(m), \Psi(m)).
	\end{align*}
Define a $2$-tensor
\begin{equation}
\label{psi}
\psi\colonequals \frac{1}{2}e_i^1\wedge e_i^2\in\wedge^2(\mathfrak{g}\oplus\mathfrak{g}),
\end{equation}
and consider the modified bivector 
\[\pi_{\text{fus}}\colonequals \pi+\psi_M.\]
Let $\Phi_1\Phi_2$ denote the pointwise product of the components of $\Phi$. Then the triple
\[(M,\pi_{\text{fus}},(\Phi_1\Phi_2,\Psi))\]
is a quasi-Poisson $G\times H$-manifold relative to the diagonal action of $G$.

Fusion equips the category of quasi-Poisson $G$-manifolds with a monoidal structure. Given two quasi-Poisson $G$-manifolds $(M_1, \pi_1,\Phi_1)$ and $(M_2, \pi_2,\Phi_2)$, their direct product $M_1\times M_2$ is naturally a quasi-Poisson manifold for the action of $G\times G$. Fusing the two sides of the $G$-action, we obtain a new quasi-Poisson $G$-manifold denoted
\[M_1\circledast M_2,\]
with bivector $(\pi_1+\pi_2)_{\text{fus}}$ and moment map $\Phi_1\Phi_2.$\\
%

%
%
%
%
%
%
%
\subsection{Nondegenerate quasi-Poisson structures}
\label{sec:nondeg}
Let $(M,\pi, \Phi)$ be a quasi-Poisson $G$-manifold. The bivector $\pi$ induces a morphism of vector bundles
\begin{align*}
\pi^{\#}\colon T^*M&\longrightarrow TM \\
	\alpha&\longmapsto \pi(\alpha,-)
	\end{align*}
from the cotangent bundle $T^*M$ to the tangent bundle $TM$. The action of $G$ differentiates to an infinitesimal action map
\[\rho\colon M\times\mathfrak{g}\longrightarrow TM.\]
The quasi-Poisson manifold $M$ is called \emph{nondegenerate} if the bundle map
\begin{align}
\label{distribution}
\pi^{\#}\oplus\rho\colon T^*M\oplus\mathfrak{g}&\longrightarrow TM\\
		(\alpha,\xi)&\longmapsto \pi^\#(\alpha)+\rho(\xi)\nonumber
			\end{align}
is surjective. For example, the double $D(G)$ defined in Example \ref{double} is nondegenerate. 

Let $\theta^L$ and $\theta^R$ be the left- and right-invariant Maurer--Cartan forms on $G$. These are $\mathfrak{g}$-valued $1$-forms defined as follows: if $L_h, R_h$ are the differentials of left- and right-multiplication by the element $h\in G$, then for any $v\in T_hG$ 
\[\theta^L_h(v)= L_{h^{-1}}v \qquad\text{and}\qquad\theta^R_h(v)=R_{h^{-1}}v.\]
The bi-invariant $3$-form on $G$ induced by $\eta\in\wedge^3\mathfrak{g}^*$, which by abusing notation we denote by the same symbol, is
\begin{equation}
\label{3form}
\eta=\frac{1}{12}\left(\theta^L,[\theta^L,\theta^L]\right)=\frac{1}{12}\left(\theta^R,[\theta^R,\theta^R]\right)\in\Gamma(\wedge^3T^*G).
\end{equation}

Every nondegenerate quasi-Poisson manifold $(M,\pi,\Phi)$ carries a (possibly degenerate, non-closed) $2$-form $\omega$ which satifies the following properties:
\begin{align*}
&\text{(Q1)}\quad d\omega=-\Phi^*\eta; \\
&\text{(Q2)}\quad \iota_{\xi_M}\omega=\frac{1}{2}\Phi^*(\theta^L+\theta^R,\xi)\quad\text{for all }\xi\in\mathfrak{g}; \\
&\text{(Q3)}\quad \ker\omega_m=\left\{\xi_M(m)\mid \xi\in\mathfrak{g} \text{ such that } \Ad_{\Phi(m)}\xi=-\xi\right\}.
\end{align*}
This $2$-form gives $M$ the structure of a \emph{quasi-Hamiltonian $G$-space} in the sense of \cite{ale.mal.mei:98}. We write $\theta_i^L,\theta_i^R\in\Gamma(T^*G)$ for the components of $\theta^L$ and $\theta^R$ in the basis $\{e_i\}$. At every point these $1$-forms are a dual basis to the left- and right-invariant vector fields, so that 
\[\theta_i^L(e_j^L)=\theta_i^R(e_j^R)=\delta_{ij}.\]
Define $C\colon TM\longrightarrow TM$ to be the morphism of vector bundles 
\begin{equation}
\label{cee}
C\colonequals \text{Id}-\frac{1}{4}\Phi^*(\theta_i^L-\theta_i^R)\otimes e_{iM}.
\end{equation}
Then $\omega$ and $\pi$ satisfy the compatibility condition
\begin{equation}
\label{dualcond}
\pi^\#\circ\omega^\flat=C,
\end{equation}
where $\omega^\flat\colon TM\longrightarrow T^*M$ is the vector bundle map given by contraction with $\omega$. 

\begin{example}
The quasi-Hamiltonian $2$-form corresponding to the nondegenerate quasi-Poisson manifold $D(G)$ from Example \ref{double} is
\[\omega=-\frac{1}{2}\left(\theta_i^{1L}\wedge \theta_i^{2R}+\theta_i^{1R}\wedge \theta_i^{2L}\right).\]
\end{example}

\begin{remark}
If the action of $G$ is trivial, the quasi-Poisson manifold $M$ is nondegenerate if and only if $\pi^\#$ is an isomorphism---that is, if and only if $\pi$ is a nondegenerate Poisson structure. In this case $\omega$ is exactly the corresponding symplectic form.
\end{remark}

Even when $\pi$ is degenerate, the image of \eqref{distribution} is an integrable generalized distribution. Its integral submanifolds, which are $G$-stable, are called the \emph{nondegenerate leaves} of $M$, because $\pi$ gives each the structure of a nondegenerate quasi-Poisson manifold. In particular, each nondegenerate leaf $S$ is equipped with a quasi-Hamiltonian $2$-form $\omega_S$. 

\begin{example}
The nondegenerate leaves of the quasi-Poisson structure $(G,\pi_G)$ defined in Example \ref{grp} are the conjugacy classes. 
\end{example}

There is an analogue of Hamiltonian reduction for quasi-Poisson manifolds. Let $(M,\pi,\Phi)$ be a quasi-Poisson $G$-manifold and fix a conjugacy class $\mathcal{O}\subset G$. If the action of $G$ on $\Phi^{-1}(\mathcal{O})$ is locally-free, then the quotient 
\[M\!\sslash_{\mathcal{O}}\!G\colonequals \Phi^{-1}(\mathcal{O})/G\]
has a natural Poisson structure whose symplectic leaves are precisely the reductions of the nondegenerate leaves of $M$. When $\mathcal{O}=\{1\}$ is the identity element, we denote this quotient  simply by $M\!\sslash\!G$.\\

%
%
%
%
%
%
%
\subsection{Twisted Dirac structures}
Fix a closed $3$-form $\phi\in\Gamma(\wedge^3T^*M)$. A vector subbundle 
\[L\subset TM\oplus T^*M\]
is called a \emph{$\phi$-twisted Dirac structure} on $M$ if it satisfies the following two conditions:
\begin{itemize}
\item $L$ is Lagrangian with respect to the symmetric pairing on $\Gamma(TM\oplus T^*M)$ given by
\[\langle(X,\alpha),(Y,\beta)\rangle=\beta(X)+\alpha(Y);\]
\item $\Gamma(L)$ is closed under the \emph{$\phi$-twisted Courant bracket} on $\Gamma(TM\oplus T^*M)$ defined by
\[\llbracket (X,\alpha),(Y,\beta)\rrbracket_\phi=([X,Y], \mathcal{L}_X\beta-\iota_Yd\alpha+\iota_{X\wedge Y}\phi).\]
\end{itemize}
The projection of $L\subset TM\oplus T^*M$ onto the first summand is an integrable generalized distribution, and induces a foliation of $M$ by \emph{presymplectic leaves}. Each presymplectic leaf $S\subset M$ carries a (potentially degenerate, non-closed) $2$-form $\omega_S$ such that $d\omega_S=\phi_{\vert S}$. 
 
\begin{example}
\begin{thmenum}
\item A symplectic structure $\omega$ on a manifold $M$ corresponds to the $0$-twisted Dirac structure
\[L_\omega\colonequals \{(X,\omega^\flat(X))\mid X\in TM\}\subset TM\oplus T^*M\]
given by the graph of $\omega^\flat$. Conversely, a $0$-twisted Dirac structure $L\subset TM\oplus T^*M$ is induced by a symplectic form if and only if $L$ is transverse to both $TM$ and $T^*M,$ viewed as subbundles of $TM\oplus T^*M$.
\item Similarly, a Poisson structure $\pi$ on $M$ corresponds to the $0$-twisted Dirac structure 
\[L_\pi\colonequals \{(\pi^\#(\alpha),\alpha)\mid \alpha\in T^*M\}\subset TM\oplus T^*M\]
given by the graph of $\pi^\#$. Its projection onto the first coordinate is the distribution whose integral submanifolds are the symplectic leaves of $\pi$. Conversely, a $0$-twisted Dirac structure $L\subset TM\oplus T^*M$ is induced by a Poisson bivector if and only if $L$ is transverse to $TM$.

\item \label{exass} \cite[Theorem 3.16]{bur.cra:05} A Hamiltonian quasi-Poisson structure $\pi$ on a $G$-manifold $M$ corresponds to the $-\Phi^*\eta$-twisted Dirac structure 
\[L=\left\{\left(\pi^\#(\alpha)+\rho(\xi), \,C^*(\alpha)+\Phi^*\sigma(\xi)\right)\mid \alpha\in T^*M, \xi\in\mathfrak{g}\right\}\subset TM\oplus T^*M.\]
Here $C$ is as defined in \eqref{cee} and $\sigma$ is given by
\begin{align}
\label{sigma}
\sigma\colon \mathfrak{g}&\longrightarrow \Gamma(T^*G) \\
		\xi&\longmapsto\frac{1}{2}\left(\xi^L+\xi^R\right)^\vee,\nonumber
		\end{align}
where $v^\vee$ is the dual of the vector field $v\in \Gamma(TG)$ under the isomorphism $TG\cong T^*G$ induced by the Killing form. 

This Dirac structure has the property that $\ker\Phi_*\cap L=0$. The associated presymplectic foliation, given by projecting $L$ onto $TM$, is exactly the foliation of $M$ by quasi-Hamiltonian leaves described in Section \ref{sec:nondeg}. 
\end{thmenum}
\end{example}

Let $(M, L_M)$ be a Dirac manifold and let $f\colon M\longrightarrow N$ be a smooth map to a manifold $N$. The \emph{pushforward} of the Dirac structure $L_M$, if it is well-defined, is the distribution
\[f_*L_M\colonequals \{(f_*X,\beta)\in TN\oplus T^*N\mid (X,f^*\beta)\in L_M\}.\]
If this distribution defines a smooth vector bundle on $N$, it is a Dirac structure. When $(M,L_M)$ and $(N,L_N)$ are Dirac manifolds, a map $f\colon M\longrightarrow N$ is \emph{forward-Dirac} if
\[L_N=f_*L_M.\]
This notion generalizes the pushforward of vector fields, and all Poisson and quasi-Poisson maps are forward-Dirac. In particular, if $(M,\pi,\Phi)$ is a quasi-Poisson $G$-manifold, then the group-valued moment map $\Phi$ is forward-Dirac when $M$ and $G$ are viewed as Dirac manifolds. Moreover, \cite[Theorem 3.16]{bur.cra:05} shows that every $\phi$-twisted Dirac manifold $(M,L)$ equipped with a forward-Dirac map $\Phi\colon M\longrightarrow G$ which satisfies
\begin{equation}
\label{strong}
\phi=-\Phi^*\eta\qquad\text{and}\qquad\ker\Phi_*\cap L=0
\end{equation}
is a quasi-Poisson manifold. (In \cite{bur.cra:09}, such a map is called \emph{strong} forward-Dirac.)

Conversely, let $(N, L_N)$ be a Dirac manifold and let $f\colon M\longrightarrow N$ be a smooth map from a manifold $M$. The \emph{pullback} of the Dirac structure $L_N$ is the Lagrangian distribution
\[f^*L_N\colonequals \{(X,f^*\beta)\in TM\oplus T^*M\mid (f_*X,\beta)\in L_N\}.\]
If this distribution defines a smooth vector bundle on $M$, it is a Dirac structure. When $(M,L_M)$ and $(N,L_N)$ are Dirac manifolds, a map $f\colon M\longrightarrow N$ is called \emph{backward-Dirac} if
\[L_M=f^*L_N.\]
This is a generalization of the pullback of differential forms---symplectomorphisms, for instance, are backward-Dirac. We give the following important example of a Dirac pullback, which we will use repeatedly in the next section.

\begin{lemma}
\label{ex:trans}
Suppose that $(N,L)$ is a $\phi$-twisted Dirac manifold. If $\imath\colon Z\hooklongrightarrow N$ is a submanifold which is transverse to the foliation of $N$ by presymplectic leaves, then
\[\imath^*L=\{(X,\imath^*\beta)\in TZ\oplus T^*Z\mid (\imath_*X,\beta)\in L\}\]
is a $\imath^*\phi$-twisted Dirac structure on $Z$. 
\end{lemma}
\begin{proof}
By \cite[Section 1.5.1]{bur:13}, the pullback $\imath^*L$ defines a Lagrangian distribution in $TZ\oplus T^*Z$. Let $p_T$ be the projection of $TN\oplus T^*N$ onto the first summand. The condition that $Z$ is transverse to the presymplectic foliation is equivalent to
\[p_T(L)_{\vert Z}+TZ=TM_{\vert Z}.\]
Then we have
\begin{align*}
L_{\vert Z}\cap TZ^\perp&=p_T(L)_{\vert Z}^\perp \cap TZ^\perp \\
			&= (p_T(L)_{\vert Z}+TZ)^\perp \\
			&=0,
\end{align*}
where the first equality follows from the fact that $L$ is Lagrangian. In particular $L_{\vert Z}\cap TZ^\perp$ has constant rank, so \cite[Proposition 1.10 and Example 1.11]{bur:13} implies that the Lagrangian distribution $\imath^*L$ is a smooth subbundle of $TZ\oplus T^*Z$.

The $\phi$-twisted Courant bracket on $TN\oplus T^*N$ restricts to a $\imath^*\phi$-twisted Courant bracket on $TZ\oplus T^*Z$. Since the Lagrangian distribution $\imath^*L$ is a smooth vector bundle, its space of sections is closed under this bracket and $\imath^*L$ is therefore a $\imath^*\phi$-twisted Dirac structure on $Z$.
\end{proof}

Note that any isomorphism which is forward-Dirac is also backward-Dirac, and vice-versa. To conclude we prove two simple ``push--pull'' lemmas which will be useful in the next section.

\begin{lemma}
\label{pushpull}
Let $A,B,C,D$ be manifolds which fit into the diagram
\begin{equation}
\label{ppull}
\begin{tikzcd}[row sep=large, column sep=large]
A	\arrow[r, "\psi"] \arrow[d, swap, "\rho"]	&B	\arrow[d, "\sigma"]\\
C	\arrow[r, "\tau"]					&D.
\end{tikzcd}
\end{equation}
Suppose that $B$ has a Dirac structure $L_B$ and that all Dirac pullbacks and pushforwards along the diagram are well-defined smooth vector bundles. If \eqref{ppull} is Cartesian, then
\[\rho_*\psi^*L_B=\tau^*\sigma_*L_B.\]
\end{lemma}
\begin{proof}
Computing, we obtain
\begin{align*}
\rho_*\psi^*L_B&=\{(\rho_*a,\gamma)\in TC\oplus T^*C\mid (a,\rho^*\gamma)\in\psi^*L_B\} \\
				&=\{(\rho_*a,\gamma)\in TC\oplus T^*C\mid \exists \beta\in T^*B\text{ such that }\psi^*\beta=\rho^*\gamma\text{ and }(\psi_*a,\beta)\in L_B\}
				\end{align*}
and 
\begin{align*}
\tau^*\sigma_*L_B&=\{(c,\tau^*\delta)\in TC\oplus T^*C\mid (\tau_*c,\delta)\in\sigma_*L_B\} \\
				&=\{(c,\tau^*\delta)\in TC\oplus T^*C\mid \exists b\in TB\text{ such that }\sigma_*b=\tau_*c\text{ and }(b,\sigma^*\delta)\in L_B\}
				\end{align*}
Suppose now that $(\rho_*a,\gamma)\in\rho_*\psi^*L_B$. Then there exists $\beta\in T^*B$ such that 
\[\psi^*\beta=\rho^*\gamma\quad\text{ and }\quad(\psi_*a,\beta)\in L_B.\]
Since the pullbacks of $\beta$ and $\gamma$ to $A$ agree, and since the diagram \eqref{ppull} is Cartesian, there exists a covector $\delta\in T^*D$ such that 
\[\sigma^*\delta=\beta\quad\text{ and }\quad\tau^*\delta=\gamma.\]
Let $b=\psi^*a$. Then $\sigma_*b=\tau_*\rho_*a$ and 
\[(b,\sigma^*\delta)=(\psi_*a,\beta)\in L_B.\]
Therefore $(\rho_*a,\gamma)\in \tau^*\sigma_*L_B,$ and we see that 
\[\rho_*\psi^*L_B\subset\tau^*\sigma_*L_B.\]
Since the two sides of this inclusion are vector bundles of the same rank over $C$, they are equal.
\end{proof}

\begin{lemma}
\label{facts}
Let $A,B,C$ be Dirac manifolds and let $\sigma\colon A\longrightarrow C$ be an isomorphism which factors as
\begin{equation*}
\begin{tikzcd}[row sep=large, column sep=large]
A	\arrow[r, hook, "\psi"] \arrow[rd, swap, "\sigma"]	&B	\arrow[d, twoheadrightarrow, "\rho"]\\
									&C,
\end{tikzcd}
\end{equation*}
where $\psi$ is backward-Dirac and $\rho$ is forward-Dirac. Then $\sigma$ is both backward- and forward-Dirac.
\end{lemma}
\begin{proof}
Write $L_A$, $L_B$, and $L_C$ for the Dirac structures on $A$, $B$, and $C$ respectively. Since $\psi$ is backward-Dirac and $\rho$ is forward-Dirac, we have
\[L_A=\psi^*L_B\quad\text{ and }\quad L_C=\rho_*L_B.\]
Since $\sigma$ is an isomorphism, it is sufficient to prove that it is forward-Dirac---that is, to show that $L_C=\sigma_*L_A.$

Suppose that $(\sigma_*a,\gamma)\in \sigma_*L_A$. Then
\begin{align*}
(a,\sigma^*\gamma)\in L_A&\quad\Rightarrow\quad (a,\psi^*\rho^*\gamma)\in L_A \\
				&\quad\Rightarrow\quad (\psi_*a,\rho^*\gamma)\in L_B\quad\text{since $L_A=\psi^*L_B$}\\
				&\quad\Rightarrow\quad (\rho_*\psi_*a,\gamma)\in L_C\quad\text{since $L_C=\rho_*L_B$}\\
				&\quad\Rightarrow\quad (\sigma_*a,\gamma)\in L_C.
				\end{align*}
Therefore $\sigma_*L_A\subset L_C$. Since $\sigma_*L_A$ and $L_C$ are vector bundles of the same rank over $C$, this inclusion implies equality. Therefore $\sigma$ is forward-Dirac. \qedhere\\
\end{proof}

%
%
%
%
%
%
%
\section{Steinberg slices}
\label{sec:stein}
In this section we show that any quasi-Poisson $G$-manifold $(M,\pi, \Phi)$ has a distinguished submanifold $M_\Sigma$ which intersects each nondegenerate leaf transversally and symplectically. This submanifold, which we call the \emph{Steinberg slice} of $M$, is the preimage of the Steinberg cross-section of $G$ under the moment map $\Phi$. It carries a Poisson structure whose symplectic leaves are its intersections with the nondegenerate leaves of $M$.

\subsection{Construction of $M_\Sigma$}
Let $W$ be the Weyl group of $G$ corresponding to a maximal torus $T$, and let $c\in W$ be a Coxeter element---that is, $c$ is the product of the simple reflections, which is unique up to conjugation. Write $\dot{c}\in N_G(T)$ for a fixed group representative of $c$.

Fix a pair of opposite Borel subgroups $B$ and $B^-$ containing $T$, and let $U$ and $U^-$ be their unipotent radicals. The \emph{Steinberg cross-section} of $G$, which was introduced in \cite{ste:65}, is the closed subvariety
\[\Sigma\colonequals U\dot{c}\cap\dot{c}\,U^-\subset G.\]
It is an affine space which consists entirely of regular elements. Its dimension is equal to the length of $c$ as an element of the Weyl group, which is the rank of $G$. 

Since $G$ is simply-connected, $\Sigma$ intersects every regular conjugacy class in $G$ exactly once and transversally \cite[Theorem 1.4]{ste:65}. If $\Xi\colon G\longrightarrow T/W$ is the quotient map induced by the Chevalley isomorphism $\mathbb{C}[G]^G\cong\mathbb{C}[T]^W$, then the composition
\begin{equation}
\label{phase}
\Sigma\hooklongrightarrow G\xlongrightarrow{\Xi} T/W
\end{equation}
is an isomorphism of affine varieties. 

\begin{lemma}
\label{vanish}
The pullback of the twisted Dirac structure $L_G$ to the cross-section $\Sigma$ is the zero Poisson structure.
\end{lemma}
\begin{proof}
Let $\jmath\colon \Sigma\longrightarrow G$ be the inclusion map. Since $\Sigma$ is transverse to the conjugacy classes of $G$, by Lemma \ref{ex:trans} the Dirac structure $L_G$ pulls back to a $\jmath^*\eta$-twisted Dirac structure 
\[L_\Sigma\colonequals \jmath^* L_G=\{(0,\jmath^*\alpha)\in T\Sigma\oplus T^*\Sigma\mid (0,\alpha)\in L_G\}\]
on $\Sigma$.

For any $h\in\Sigma$, the image of 
\[\theta^R_h= (R_{h^{-1}})_*\colon T_h\Sigma\longrightarrow \mathfrak{g}\]
is contained in $\mathfrak{b}=\text{Lie}(B)$, and $(\mathfrak{b},[\mathfrak{b},\mathfrak{b}])=0$. In view of \eqref{3form} we have $\jmath^*\eta=0$, so the Dirac structure $L_\Sigma$ is non-twisted. 

Moreover, since $L_\Sigma$ is a Lagrangian subbundle of $T\Sigma\oplus T^*\Sigma$, it follows that the map $\jmath^*$ has full rank and therefore
\[L_\Sigma=\{(0,\alpha)\in T\Sigma\oplus T^*\Sigma\mid \alpha\in T^*\Sigma\}\]
corresponds to the zero Poisson structure.
\end{proof}

\begin{theorem}
\label{slices}
Let $H$ and $K$ be complex groups. Suppose that $K$ is simply-connected and semisimple, and let $\Sigma$ be the Steinberg cross-section of $K$. Let $(M,\pi)$ be a quasi-Poisson $H\times K$-manifold with group-valued moment map
\begin{align*}
\Phi\colon M&\longrightarrow H\times K \\
	m&\longmapsto (\Phi_H(m), \Phi_K(m)).
	\end{align*}
\begin{enumerate}[label=\emph{(\alph*)}]
\item The preimage $M_{H,\Sigma}\colonequals \Phi^{-1}_K(\Sigma)$ is a smooth submanifold of $M$.
\item The pullback of the twisted Dirac structure $L_M$ to $M_{H,\Sigma}$ is quasi-Poisson for the action of $H$, with group-valued moment map $\Phi_{H\vert M_{H,\Sigma}}$. 
\item The nondegenerate leaves of $M_{H,\Sigma}$ are the connected components of $M_{H,\Sigma}\cap S$, where $S$ varies over all nondegenerate leaves of $M$; the quasi-Hamiltonian $2$-form on each connected component of $M_{H,\Sigma}\cap S$ is the restriction of the quasi-Hamiltonian form $\omega_S$.
\end{enumerate}
\end{theorem}

\begin{proof}
(a) Let $k\in \Sigma$ and $m\in\Phi_K^{-1}(k)$, and write $\mathcal{O}\subset K$ for the conjugacy class of $k$. Because $\Phi_K$ is $K$-equivariant,
\[T_k\mathcal{O}=\Phi_{K*}(T_m(K\cdot m)) \subset \Phi_{K*}(T_mM).\]
Therefore, since $\Sigma$ is transverse to $\mathcal{O}$, it is transverse to $\Phi_K$. It follows that $M_{H,\Sigma}=\Phi_K^{-1}(\Sigma)$ is a smooth submanifold of $M$.

(b) Let $\jmath\colon H\times\Sigma\hooklongrightarrow H\times K$ be the inclusion, and let $\eta=(\eta_H,\eta_K)$ be the canonical $3$-form on the product $H\times K$. Since $\Sigma$ is transverse to the conjugacy classes of $K$, 
\[T_mM_{H,\Sigma}+ T_m(K\cdot m)=\Phi_{K*}^{-1}(T_k\Sigma+T_k\mathcal{O})=T_mM.\]
It follows that $M_{H,\Sigma}$ is transverse to the $H\times K$-orbits on $M$, and therefore also to the presymplectic leaves of the Dirac structure $L_M$. 

Let $\imath\colon M_{H,\Sigma}\hooklongrightarrow M$ be the inclusion. By Example \ref{ex:trans}, transversality implies that the $-\Phi^*\eta$-twisted Dirac structure on $M$ pulls back to a $-\imath^*(\Phi^*\eta)$-twisted Dirac structure $L_{M_{H,\Sigma}}$ on $M_{H,\Sigma}$. The commutative diagram
\begin{equation}
\label{fiberdiag}
\begin{tikzcd}[row sep=large]
M_{H,\Sigma}	\arrow[r, hook, "\imath"] \arrow{d}	{\Phi}&M	\arrow{d}{\Phi}\\
H\times \Sigma			\arrow[r, hook, "\jmath"]			&H\times K
\end{tikzcd}
\end{equation}
shows that 
\[-\imath^*(\Phi^*\eta)=-\Phi^*(\jmath^*\eta)=-\Phi_{H*}\eta_H,\]
where the last equality follows since the restriction of $\eta_K$ to $\Sigma$ vanishes by Lemma \ref{vanish}. Therefore $L_{M_{H,\Sigma}}$ is a $-\Phi_{H*}\eta_H$-twisted Dirac structure on $M_{H,\Sigma}$.

To show that $L_{M_{H,\Sigma}}$ is quasi-Poisson for the action of $H$, by \eqref{strong} it is sufficient to show (i) that $\Phi_H$ is a forward-Dirac map and (ii) that it has the property
\[\ker\Phi_{H*}\cap L_{M_{H,\Sigma}}=0.\]

(i) Let $p_H\colon H\times K\longrightarrow H$ be the first projection, so that $\Phi_H=p_H\circ\Phi$. Since $p_H$ and $\Phi$ are forward-Dirac, we obtain
\begin{align*}
L_H&=p_{H*}L_{H\times\Sigma} \\
	&=p_{H*}(\jmath^*\Phi_*L_M) \\
	&=p_{H*}(\Phi_*\imath^*L_M) =\Phi_{H*}L_{M_{H,\Sigma}},
	\end{align*}
where the third equality follows from applying Lemma \ref{pushpull} to \eqref{fiberdiag}. Therefore $\Phi_H$ is forward-Dirac.

(ii) Let $p_K\colon H\times K\longrightarrow K$ be the second projection, so that 
\[L_\Sigma=p_{K*}L_{H\times\Sigma}.\]
Suppose that $(X,0)\in L_{M_{H,\Sigma}}$ is such that $\Phi_{H*}(X)=0$. In particular, it follows that $(X,0)\in L_M$. Since $L_\Sigma$ is the zero Poisson by Lemma \ref{vanish} and since $p_K$ and $\Phi$ are forward-Dirac, 
\[\Phi_{K*}(X)=p_{K*}\Phi_*(X)=0.\]
This implies that $\Phi_*(X)=0$. Therefore $X\in\ker\Phi_*\cap L_M$. Since $\Phi$ satisfies condition \eqref{strong}, it follows that $X=0$.

(c) This is immediate since we have shown that the quasi-Poisson structure on $M_{H,\Sigma}$ is the pullback of the Dirac structure $L_M$.
\end{proof}

When $H=1$ and $K=G$, Theorem \ref{slices} has the following corollary:

\begin{corollary}
\label{slice-cor}
Let $(M,\pi,\Phi)$ be a quasi-Poisson $G$-manifold. Then
\begin{enumerate}[label=\emph{(\alph*)}]
\item $M_\Sigma\colonequals \Phi^{-1}(\Sigma)$ is a smooth submanifold of $M$.
\item The pullback of the twisted Dirac structure $L_M$ to $M_\Sigma$ is Poisson.
\item The symplectic leaves of $M_\Sigma$ are the connected components of $M_\Sigma\cap S$, where $S$ varies over all nondegenerate leaves of $M$; the symplectic form on each connected component of $M_\Sigma\cap S$ is the restriction of the quasi-Hamiltonian $2$-form $\omega_S$.
\end{enumerate}
\end{corollary}

Our first example of a Steinberg slice is the group scheme of regular centralizers of $G$, whose symplectic structure is constructed in essentially the same way in \cite[Section 2]{fin.tsy:19}.

\begin{example}
\label{ex:double}
Consider the double $D(G)$ of Example \ref{double}. Recall that its moment map is
\begin{align*}
\mu\colon D(G)&\longrightarrow G\times G \\
	(a,b)&\longmapsto (aba^{-1}, b^{-1}),
	\end{align*}
with image
\begin{equation}
\label{image}
\im(\mu)=\left\{(g,h)\in G\times G\mid g \text{ is conjugate to } h^{-1}\right\}.
\end{equation}
Let
\[\Sigma_\Delta\colonequals \left\{(h,h^{-1})\mid h\in\Sigma\right\}\subset G\times G\]
be the antidiagonal embedding of the Steinberg cross-section $\Sigma$. Since two elements of $\Sigma$ are conjugate if and only if they are equal, we have
\[\mu^{-1}(\Sigma_\Delta)=\mu^{-1}(\Sigma\times \iota(\Sigma)),\]
where $\iota\colon G\longrightarrow G$ is the inversion. Since $\Sigma\times\iota(\Sigma)$ is a Steinberg cross-section in $G\times G$, it follows from Corollary \ref{slice-cor} that $\mu^{-1}(\Sigma_\Delta)$ is a smooth submanifold of $D(G)$ with an induced symplectic structure.

The fiber of $\mu$ above an antidiagonal point $(h,h^{-1})\in G\times G$ is the $G$-centralizer of $h$, and therefore
\[\mu^{-1}(\Sigma_\Delta)=\{(a,h)\in G\times\Sigma \mid aha^{-1}=h\}.\]
This space is the completion of the phase space of the open relativistic Toda lattice, and this symplectic structure is precisely the one described in \cite[Lemma 2.1]{fin.tsy:19}. \\
\end{example}

%
%
%
%
%
%
%
\subsection{Slices and the base affine space}
We may also take the preimage of the Steinberg cross-section through only one component of the moment map \eqref{moment}. This is the analogue of the one-sided Whittaker reduction of $T^*G$, which gives the twisted cotangent bundle of the base affine space $G/U$.

\begin{example}
By letting $M=D(G)$ and $H=K=G$ in Theorem \ref{slices}, we see that
\[D_\Sigma(G)\colonequals G\times\iota(\Sigma),\]
has a natural nondegenerate quasi-Poisson structure for the residual $G$-action 
\[g\cdot(a,h)=(ga,h),\qquad\text{for $g\in G$, $(a,h)\in D_\Sigma(G)$}.\]
The corresponding group-valued moment map is
\begin{align*}
D_\Sigma(G)&\longrightarrow G\\
	(a,h)&\longrightarrow aha^{-1}.
	\end{align*}
\end{example}

\begin{remark}
Consider the affine space $\Theta\colonequals U\dot{c}U$, which contains $\Sigma$. By \cite[Section 8.9]{ste:65}, the conjugation action gives an isomorphism 
\[U\times\Sigma\xlongrightarrow{\sim} \Theta.\]
(The proof in \emph{loc. cit.} is omitted, but more general versions of this statement are proved in \cite[Proposition 2.1]{sev:11} or \cite[Theorem 3.6]{lus.he:12}.) Then $D_\Sigma(G)$ becomes a bundle of affine spaces
\[D_\Sigma(G)\cong G\times_U\iota(\Theta)\longrightarrow G/U.\]
\end{remark}

The one-sided slice $D_\Sigma(G)$ is similar to the universal imploded cross-section of \cite{hur.jef.sja:06}, where the authors study real quasi-Hamiltonian manifolds under the action of compact Lie groups. Following this analogy, we will show that the Steinberg slice $M_\Sigma$ can always be obtained as a quasi-Poisson reduction of the fusion product $M\circledast D_\Sigma(G)$. We will need the following two preliminary lemmas.

\begin{lemma}
\label{aux1}
Let $(N,\pi,\Phi)$ be a nondegenerate quasi-Poisson $G$-manifold with quasi-Hamiltonian $2$-form $\omega$ and Dirac structure $L_N$. Suppose that 
\[g\colon M\longrightarrow N\]
is a map from a manifold $M$ which has the property that $g^*\omega=0$. Then $g^*L_N=TM\oplus 0$.
\end{lemma}
\begin{proof}
Let $(X,g^*\gamma)$ be any element of $g^*L_N$, so that $(g_*X,\gamma)\in L_N$. Then there is a $1$-form $\alpha\in T^*N$ and a Lie algebra element $\xi\in\mathfrak{g}$ such that 
\[g_*X=\pi^\#(\alpha)+\rho(\xi)\qquad\text{and}\qquad\gamma=C^*(\alpha)+\Phi^*\sigma(\xi),\]
where $\rho$ is the infinitesimal action map, $C^*$ is as defined in \eqref{cee}, and $\sigma$ is given by \eqref{sigma}.

The assumption that $g^*\omega=0$ implies that 
\begin{align*}
0&=g^*\omega^\flat(\pi^\#(\alpha)+\rho(\xi))\\
	&=g^*C^*(\alpha)+g^*\omega^\flat(\rho(\xi))\\
	&=g^*C^*(\alpha)+g^*\Phi^*(\sigma(\xi))\\
	&=g^*\gamma.
	\end{align*}
Here the second equality follows by taking the dual of \eqref{dualcond}, and the third equality follows from the quasi-Hamiltonian moment map condition (Q2). We conclude that $g^*L_N\subset TM\oplus 0$, and since these are vector bundles of the same rank they must therefore be equal.
\end{proof}

\begin{lemma}
\label{aux2}
Let $(M,\pi_M,\Phi_M)$ be a quasi-Poisson $G$-manifold with corresponding Dirac structure $L_M$, and let $(N,\pi_N,(\Phi_N,\Psi))$ be a quasi-Poisson $G\times H$-manifold with Dirac structure $L_N$. Suppose that $g\colon M\longrightarrow N$ is a map which satisfies
\begin{equation}
\label{assum}
g^*L_N=TM\oplus0\qquad\text{and}\qquad(\Phi_N\circ g)(m)=\Phi_M(m)^{-1}\text{ for all }m\in M.
\end{equation}
Then the embedding
\begin{align*}
f\colon M&\hooklongrightarrow M\circledast N\\
m&\longmapsto (m,g(m))
\end{align*}
is backward-Dirac.
\end{lemma}
\begin{proof}
Let $L_{M\circledast N}$ be the Dirac structure associated to $M\circledast N$. We will show that $L_M=f^*L_{M\circledast N}$, by showing that the right-hand side is contained in the left-hand side.

Write 
\begin{align*}
&\rho_M\colon M\times \mathfrak{g}\longrightarrow TM\\
&\rho_N\colon N\times \mathfrak{g}\longrightarrow TN\\
&\tau\colon N\times \mathfrak{h}\longrightarrow TN\\
&\rho\colon G\times\mathfrak{g}\longrightarrow TG
\end{align*}
for the infinitesimal action maps, and let
\[C^*_{M\circledast N}\colon T^*(M\circledast N)\longrightarrow T^*(M\circledast N)\]
be the map of vector bundles defined in \eqref{cee}. By Example \ref{exass}, the Dirac structure $f^*L_{M\circledast N}$ consists of all pairs $(X, f^*\gamma)\in TM\oplus T^*M$ such that
\[X+g_*X=\pi_M^\#(\alpha)+\pi^\#_N(\beta)+\psi_{M\times N}^\#(\alpha+\beta)+\rho_M(\xi)+\rho_N(\xi)+\tau(\chi)\]
and 
\[\gamma= C^*_{M\circledast N}(\alpha+\beta)+(\Phi_M\Phi_N)^*\sigma_G(\xi)+\Psi^*\sigma_H(\chi),\]
where $\alpha\in T^*M$, $\beta\in T^*N$, $\xi\in\mathfrak{g}$, and $\chi\in\mathfrak{h}$. Here $\sigma_G$ and $\sigma_H$ are defined as in \eqref{sigma}, and the bivector $\psi_{M\times N}$ is induced by \eqref{psi}. In particular, the second condition of \eqref{assum} implies that $f^*\circ(\Phi_M\Phi_N)^*=0$, and therefore that
\begin{equation}
\label{cancels}
f^*\gamma=\alpha+g^*\beta-\frac{1}{4}\beta(f_{iN})g^*\Psi^*(\zeta_i^L-\zeta_i^R)+g^*\Psi^*\sigma_H(\chi),
\end{equation}
where $\{f_i\}$ is an orthonormal basis of $\mathfrak{h}$ and $\{\zeta_i\}$ is the corresponding dual basis in $\mathfrak{h}^*$.

Summing once again over repeated indices, define 
\[\xi_\alpha\colonequals \frac{1}{2}\alpha(e_{iM})e_{i}\qquad\text{and}\qquad\xi_\beta\colonequals \frac{1}{2}\beta(e_{iN})e_i.\]
Then 
\[X=\pi_M^\#(\alpha)+\rho_M(\xi-\xi_\beta)\qquad\text{and}\qquad g_*X=\pi^\#_N(\beta)+\rho_N(\xi+\xi_\alpha)+\tau(\chi).\]
It then follows from the first condition of \eqref{assum} that
\[g^*(C^*_N(\beta)+\Phi^*_N\sigma_G\left(\xi+\xi_\alpha\right)+\Psi^*\sigma_H(\chi))=0.\]
Applying this to equation \eqref{cancels} gives
\begin{align*}
f^*\gamma&=\alpha+\frac{1}{4}\beta(e_{iN})g^*\Phi_N^*(\theta_i^L-\theta_i^R)-g^*\Phi_N^*\sigma_G(\xi+\xi_\alpha)\\
		&=\alpha+\frac{1}{4}\beta(e_{iN})\Phi_M^*(\theta_i^L-\theta_i^R)+\Phi_M^*\sigma_G(\xi+\xi_\alpha).
		\end{align*}
Here the last equality follows once again from the second condition of \eqref{assum}, which implies that 
\[g^*\Phi_N^*(\theta_i^L)=-\Phi_M^*(\theta_i^R)\qquad\text{and}\qquad g^*\Phi_N^*(\theta_i^R)=-\Phi_M^*(\theta_i^L).\]

To complete the proof, we will now show that the $1$-forms $\alpha$ and $\beta$ satisfy the symmetry relation
\begin{equation}
\label{symmetry}
\alpha(e_{iM})=-\beta(e_{iN})\qquad\text{for all basis elements }e_i\in\mathfrak{g}.
\end{equation}
Since the quasi-Poisson moment maps $\Phi_M$ and $\Phi_N$ are $G$-equivariant, 
\begin{align*}
R_{\Phi_M^{-1}}\Phi_{M*}\rho_M(\xi)&=R_{\Phi_M^{-1}}\rho(\xi)\\
							&=\xi-(\Ad_{\Phi_M}\xi)\\
							&=-L_{\Phi_M}\rho(\xi)=-L_{\Phi_M}\Phi_{N*}\rho_N(\xi).
							\end{align*}
We therefore obtain
\begin{align}
\label{abc}
0&=(\Phi_M\Phi_N)_*(X+g_*X)\\
	&=(R_{\Phi_N}\Phi_{M*}+L_{\Phi_M}\Phi_{N*})(X+g_*X) \nonumber\\
	&=R_{\Phi_M^{-1}}\Phi_{M*}\pi_M^\#(\alpha)+R_{\Phi_M^{-1}}\rho(-\xi_\beta)+L_{\Phi_M}\Phi_{N*}\pi_N^\#(\beta)+L_{\Phi_M}\rho(\xi_\alpha)\nonumber
	\end{align}
We unpack the last line term-by-term. Using the quasi-Poisson moment map condition \cite[Lemma 2.3]{ale.kos.mei:02}, 
\[R_{\Phi_M^{-1}}\Phi_{M*}\pi_M^\#(\alpha)(\theta_i)=-\alpha\circ \rho_M\circ\sigma^\vee(\theta_i^R)=-\frac{1}{2}\alpha((\Ad_{\Phi_M^{-1}}e_i)_M+e_{iM})\]
Summing over repeated indices, we conclude that 
\[R_{\Phi_M^{-1}}\Phi_{M*}\pi_M^\#(\alpha)=-\frac{1}{2}\alpha((\Ad_{\Phi_M^{-1}}e_i)_M)e_i-\frac{1}{2}\alpha(e_{iM})e_i=-(\Ad_{\Phi_M}\xi_\alpha+\xi_\alpha),\]
where the second equality follows from the fact that $\Ad_{\Phi_M}$ is an orthogonal operator. Similarly,
\[L_{\Phi_M}\Phi_{N*}\pi_N^\#(\beta)=-(\Ad_{\Phi_M}\xi_\beta+\xi_\beta).\]
Moreover,
\[L_{\Phi_M}\rho(\xi_\alpha)=\Ad_{\Phi_M}\xi_\alpha-\xi_\alpha\qquad\text{and}\qquad R_{\Phi_M^{-1}}\rho(-\xi_\beta)=\Ad_{\Phi_M}\xi_\beta-\xi_\beta.\]
Equation \eqref{abc} therefore becomes
\[\xi_\alpha=-\xi_\beta,\]
proving the claim \eqref{symmetry}.

Now we return to our computation of the pullback of $\gamma$, which in view of \eqref{symmetry} gives
\begin{align*}
f^*\gamma&=\alpha-\frac{1}{4}\alpha(e_{iM})\Phi_M^*(\theta_i^L-\theta_i^R)+\Phi_M^*\sigma_G(\xi-\xi_\beta)\\
		&=C^*_M(\alpha)+\Phi_M^*\sigma_G(\xi-\xi_\beta).
		\end{align*}
This shows that $(X,f^*\gamma)\in L_M$, proving the desired inclusion and implying that $f$ is backward-Dirac.
\end{proof}

Now let $(M,\pi,\Phi)$ be a quasi-Poisson $G$-manifold, and consider the embedding
\begin{align*}
f\colon M&\hooklongrightarrow M\circledast D(G) \\
m&\longmapsto (m,1,\Phi(m)^{-1}),
\end{align*}
where $M\circledast D(G)$ is given by fusing the $G$-action on $M$ with the first $G$-action on $D(G)$. In view of \cite[Remark 3.2]{ale.mal.mei:98}, the pullback of the quasi-Hamiltonian $2$-form on $D(G)$ to the submanifold $\{1\}\times G$ vanishes. By Lemma \ref{aux1}, $f$ satisfies the conditions of Lemma \ref{aux2} and is therefore a backward-Dirac map.

The restriction of $f$ to $M_\Sigma$ descends to a backward-Dirac embedding
\begin{equation}
\label{embs}
M_\Sigma\hooklongrightarrow M\circledast D_\Sigma(G).
\end{equation}
The moment map of the right-hand side is given by
\begin{align*}
J\colon M\circledast D_\Sigma(G)&\longrightarrow G \\
(m,a,x)&\longrightarrow \Phi(m)axa^{-1},
\end{align*}
and the image of \eqref{embs} is contained in the fiber $J^{-1}(1)$ above the identity.

\begin{proposition}
The embedding \eqref{embs} induces an isomorphism of Poisson manifolds
\[M_\Sigma\cong (M\circledast D_\Sigma(G))\!\sslash\!G.\]
\end{proposition}
\begin{proof}
Since the diagonal action of $G$ on $M\circledast D_\Sigma(G)$ is free, each $G$-orbit in $J^{-1}(1)$ contains a unique element of the form $(m,1,\Phi(m)^{-1})$. Therefore the induced map
\begin{equation*}
M_\Sigma\longrightarrow J^{-1}(1)/G=(M\circledast D_\Sigma(G))\!\sslash\!G
\end{equation*}
is an isomorphism. We only need to check that it is Poisson. 

Since $J$ is a quasi-Poisson moment map, it satisfies the transversality condition \eqref{strong}. Therefore the inclusion 
\begin{equation}
\label{incl}
J^{-1}(1)\hooklongrightarrow M\circledast D_\Sigma(G)
\end{equation}
equips $J^{-1}(1)$ with a Dirac structure via pullback, and we get a diagram
\begin{equation*}
\begin{tikzcd}[row sep=large]
M_\Sigma	\arrow[r, hook] \arrow[rd, swap, "\sim"]		&J^{-1}(1)\arrow[d, twoheadrightarrow]\\
&J^{-1}(1)/G.
\end{tikzcd}
\end{equation*}
Here the horizontal arrow is backward-Dirac because its composition with \eqref{incl} is the backward-Dirac map \eqref{embs}. The vertical arrow is forward-Dirac by \cite[Theorem 4.11]{bur.cra:05} and, since the diagonal arrow is an isomorphism, Lemma \ref{facts} implies that it is forward-Dirac, and therefore Poisson.\qedhere\\
\end{proof}

%
%
%
%
%
%
%
\subsection{Log-nondegenerate quasi-Poisson structures}
\label{sec:lognon}
Once again let $(M,\pi,\Phi)$ be a quasi-Poisson $G$-manifold, and let $D\subset M$ be a $G$-stable divisor with simple normal crossings. The \emph{logarithmic tangent sheaf} is the sheaf of logarithmic vector fields on $M$---that is, vector fields which are tangent to the divisor $D$. Because $D$ has simple normal crossings, this sheaf is locally free. The associated vector bundle $T_DM$ is called the \emph{logarithmic tangent bundle} of $M$, and its dual $T^*_DM$ is the \emph{logarithmic cotangent bundle}.

Suppose for the rest of this section that the bivector field $\pi$ is logarithmic. Then it corresponds to a natural morphism of vector bundles
\[\pi^\#_D\colon T^*_DM\longrightarrow T_DM\]
from the log-cotangent bundle of $M$ to the log-tangent bundle. Similarly, any logarithmic $2$-form $\omega\in\Gamma(\wedge^2T^*_DM)$ corresponds to
\[\omega_D^\flat\colon T_DM\longrightarrow T_D^*M.\]
Since the action of $G$ on $M$ stabilizes the divisor $D$, there is also a logarithmic infinitesimal action map
\[\rho_D\colon M\times\mathfrak{g}\longrightarrow T_DM.\]

\begin{definition}
The quasi-Poisson $G$-manifold $M$ is \emph{log-nondegenerate} if the morphism of vector bundles
\begin{align}
\label{lognon}
\pi^\#_D\oplus\rho_D\colon T^*_DM\oplus\mathfrak{g}&\longrightarrow T_DM\\
				(\alpha,\xi)&\longmapsto \pi_D^\#(\alpha)+\rho_D(\xi)\nonumber
\end{align}
is surjective.
\end{definition}

\begin{remark}
The pullback of $T_DM$ to the open dense locus $M^\circ\colonequals M\backslash D$ is just the ordinary cotangent bundle $TM^\circ$; similarly, the pullback of $T^*_DM$ to $M^\circ$ is $T^*M^{\circ}$. Therefore, along $M^\circ$ the morphism $\eqref{lognon}$ agrees with the morphism of vector bundles \eqref{distribution}. In particular, if $M$ is log-nondegenerate then $M^\circ$ is its unique open dense nondegenerate leaf.
\end{remark}

Viewed as an automorphism of the tangent sheaf of $M$, the map defined in \eqref{cee} takes logarithmic vector fields to logarithmic vector fields. Therefore it defines a morphism of vector bundles
\begin{equation}
\label{cee2}
C_D\colon T_DM\longrightarrow T_DM.
\end{equation}
Using this and \cite[Theorem 10.3]{ale.kos.mei:02}, we give an equivalent condition for log-nondegeneracy.

\begin{proposition}
\label{lognoncrit}
The quasi-Poisson manifold $(M,\pi,\Phi)$ is log-nondegenerate if and only if there exists a logarithmic $2$-form $\omega\in \Gamma(\wedge^2T_D^*M)$ such that 
\begin{equation}
\label{equality}
\pi_D^\#\circ\omega_D^\flat=C_D.
\end{equation}
\end{proposition}
\begin{proof} ($\Rightarrow$) First suppose that $\pi$ is log-nondegenerate. Then its restriction to $M^\circ$ is a nondegenerate quasi-Poisson bivector $\pi^\circ$. By \cite[Theorem 10.3]{ale.kos.mei:02} there is a $2$-form 
\[\omega^\circ\in\Gamma(\wedge^2T^*M^\circ)\]
which satisfies conditions (Q1), (Q2), and (Q3). If $C^\circ$ is the restriction of \eqref{cee2} to $M^\circ$, then
\[\pi^{\circ\#}\circ\omega^{\circ\flat}=C^\circ.\]
Taking duals, we also obtain
\begin{equation}
\label{comps}
\omega^{\circ\flat}\circ\pi^{\circ\#}=C^{\circ*}.
\end{equation}
If $\omega^\circ$ extends to a logarithmic $2$-form $\omega$ on $M$, then condition \eqref{equality} is automatically satisfied by continuity. Therefore it is enough to show that $\omega^{\circ\flat}\colon TM^\circ\longrightarrow T^*M^\circ$ extends to a morphism of vector bundles 
\[\omega_D^\flat\colon T_DM\longrightarrow T^*_DM.\]

First, define a morphism of vector bundles $\lambda_D\colon  T^*_DM\oplus\mathfrak{g}\longrightarrow T^*_DM$ by
\[\lambda_D(\alpha,e_i)=C^*_D(\alpha)+\frac{1}{2}\Phi^*(\theta_i^L+\theta_i^R),\]
and let $\lambda^\circ\colon T^*M^\circ\oplus\mathfrak{g}\longrightarrow T^*M^\circ$ be its restriction to $M^\circ$. In view of condition (Q2) and equation \eqref{comps}, this restriction descends to the morphism $\omega^{\circ\flat}$ through the commutative diagram
\begin{equation*}
\begin{tikzcd}[row sep=large, column sep=large]
T^*M^\circ\oplus\mathfrak{g}	\arrow[r, "\lambda^\circ"]\arrow[d, swap, twoheadrightarrow, "\pi^{\circ\#}\oplus\rho"]		&T^*M^\circ\\
TM^\circ.\arrow[ru, swap, "\omega^{\circ\flat}"]&
\end{tikzcd}
\end{equation*}

Second, for any local section $\alpha$ of $T_D^*M$ let
\[\omega_D^\flat(\pi_D^\#(\alpha))\colonequals C^*_D(\alpha).\]
This is well-defined: if $\pi_D^\#(\alpha)=0$, it follows that $\pi^{\circ\#}$ vanishes on the restriction $\alpha_{\vert M^\circ}$, so that $C^{\circ*}(\alpha_{\vert M^\circ})=0$ and therefore $C^*_D(\alpha)=0$. Therefore this defines $\omega_D^\flat$ on the image of $\pi^\#_D$. On the other hand, the condition (Q2) defines $\omega_D^\flat$ on the image of $\rho_D$. By the log-nondegeneracy assumption \eqref{lognon}, this determines $\omega_D^\flat$ entirely.

Since $\omega^{\circ\flat}$ is well-defined on the intersection of the images of $\pi^{\circ\#}$ and $\rho$, it follows that $\omega_D^\flat$ is also well-defined on 
\[\im(\pi_D^\#)\cap\im(\rho_D).\]
Therefore the map $\omega_D^\flat$ fits into the following commutative diagram:
\begin{equation*}
\begin{tikzcd}[row sep=large, column sep=large]
T_D^*M\oplus\mathfrak{g}	\arrow[r, "\lambda_D"]\arrow[d, swap, twoheadrightarrow, "\pi_D^{\#}\oplus\rho_D"]		&T_D^*M\\
T_DM.\arrow[ru, swap, "\omega_D^{\flat}"]&
\end{tikzcd}
\end{equation*}
In particular, this implies that $\omega_D^\flat$ is a smooth morphism of vector bundles.

($\Leftarrow$) Conversely, suppose that there exists a logarithmic $2$-form $\omega$ on $M$ such that \eqref{equality} holds, and let $v\in T_DM$ be any logarithmic vector. Then, in view of \eqref{cee},
\[\pi_D^\#\circ\omega^\flat_D(v)=C_D(v)=v-\rho_D(\xi)\]
for some $\xi\in\mathfrak{g}$. It follows that
\[\pi_D^\#\left(\omega^\flat_D(v)\right)+\rho_D(\xi)=v,\]
and so $\pi^\#_D\oplus\rho_D$ is surjective. Therefore $\pi$ is log-nondegenerate.
\end{proof}

\begin{remark}
Together with \cite[Theorem 10.3]{ale.kos.mei:02}, Proposition \ref{lognoncrit} implies that any log-nondegenerate quasi-Poisson manifold comes equipped with a unique logarithmic $2$-form which satisfies logarithmic versions of conditions (Q1), (Q2), (Q3), as well as the compatibility condition \eqref{equality}. 

In the special case that the action of $G$ is trivial, $(M,\pi)$ is log-nondegenerate if and only if $\pi_D^\#$ is an isomorphism---that is, if and only if $\pi$ is a log-symplectic Poisson structure. In this case $C_D$ is the identity morphism and the logarithmic $2$-form $\omega_D$ is is exactly the corresponding log-symplectic form.
\end{remark}

The following proposition shows that Steinberg slices in log-nondegenerate quasi-Poisson manifolds are log-symplectic.

\begin{proposition}
\label{prop:stein-lognon}
Suppose that $(M,\pi,\Phi)$ is log-nondegenerate. 
\begin{enumerate}[label=\emph{(\alph*)}]
\item $M_\Sigma\cap D$ is a simple normal crossing divisor in $M_\Sigma$.
\item The induced bivector $\pi_\Sigma$ is tangent to $M_\Sigma\cap D$.
\item $(M_\Sigma,\pi_\Sigma)$ is a log-symplectic Poisson manifold.
\end{enumerate}
\end{proposition}
\begin{proof} 
(a) Let $D_1,\ldots, D_l$ be the smooth irreducible components of the simple normal crossing divisor $D$. Since the bivector $\pi$ is tangent to $D$ and since $D$ is $G$-stable, each partial intersection 
\[\qquad\bigcap_{i\in I}D_i,\qquad I\subset\{1,\ldots,l\}\]
is a union of nondegenerate leaves of $(M,\pi)$. Since $M_\Sigma$ is transverse to these nondegenerate leaves, it is transverse to every partial intersection of divisor components. It follows that $M_\Sigma\cap D$ is again a simple normal crossing divisor.

(b) Fix a point $m\in M_\Sigma$ and a covector $\alpha\in T^*_mM_\Sigma$, and let $\imath\colon M_\Sigma\longrightarrow M$ be the inclusion map. Write $L_{M_\Sigma}$ and $L_M$ for the twisted Dirac structures associated to $M_\Sigma$ and $M$. By Theorem \ref{slices},
\[L_{M_\Sigma}=\imath^*L_M.\]
Therefore, since $(\pi_\Sigma^\#(\alpha),\alpha)\in L_{M_\Sigma}$, there exists some $\beta\in T^*_mM$ such that
\[\left(\pi_\Sigma^\#(\alpha),\alpha\right)=\left(\pi_\Sigma^\#(\alpha),\imath^*\beta\right)\qquad\text{and}\qquad \left(\imath_*\pi_\Sigma^\#(\alpha),\beta\right)\in L_M.\]
Since $(M,\pi)$ is quasi-Poisson, Example \ref{exass} then implies that
\[\imath_*\pi_\Sigma^\#(\alpha)=\pi^\#(\gamma)+\rho(\xi)\]
for some $\gamma\in T_m^*M$ and $\xi\in\mathfrak{g}$. Since $\pi$ is logarithmic and $D$ is $G$-stable, both terms on the right-hand side are tangent to $D$. It follows that $\pi_\Sigma^\#(\alpha)$ is tangent to $M_\Sigma\cap D$, and therefore the bivector $\pi_\Sigma$ is logarithmic.

(c) Let $\omega$ be the logarithmic $2$-form on $M$ defined by Proposition \ref{lognoncrit}. Write $\omega_\Sigma$ for its restriction to $M_\Sigma$, and $\omega_\Sigma^\circ$ for its restriction to $M_\Sigma^\circ\colonequals M_\Sigma\cap M^\circ$. Since $(M^\circ,\pi^\circ)$ is nondegenerate and $M_\Sigma^\circ\subset M^\circ$ is a Steinberg slice, it follows from Theorem \ref{slices} that $\omega_\Sigma^\circ$ is a symplectic form. Therefore
\[\pi_{\Sigma}^{\circ\#}\circ\omega_{\Sigma}^{\circ\flat}\colon  TM_\Sigma^\circ\longrightarrow TM_\Sigma^\circ\]
is the identity map. 

There is a morphism of vector bundles
\[\pi_{\Sigma,D}^\#\circ\omega_{\Sigma,D}^\flat\colon  T_DM_\Sigma\longrightarrow T_DM_\Sigma.\]
For simplicity and since there is no risk of confusion, here we abuse notation to write $T_DM_\Sigma$ for the log-tangent bundle of $M_\Sigma$ relative to the normal crossing divisor $M_\Sigma\cap D$. This morphism agrees with the identity map along $M_\Sigma^\circ$. Therefore it agrees with the identity map everywhere, and $\pi_\Sigma$ is log-symplectic.\qedhere\\
\end{proof}

%
%
%
%
%
%
%
\section{The multiplicative universal centralizer and the wonderful compactification}
\label{sec:mid}
Let $Z_G$ be the center of the simply-connected, semisimple group $G$, and let $G_{\ad}\colonequals G/Z_G$ be its adjoint form. A finite quotient of Example \ref{ex:double} produces a smooth, symplectic family of centralizer subgroups of $G_{\ad}$ over $\Sigma$. In the next sections we will compactify the centralizer fibers of this family inside the wonderful compactification of $G_{\ad}$. First we recall the construction of this universal centralizer and of the wonderful compactification.

\subsection{The multiplicative universal centralizer}
The natural action of $G$ on itself by conjugation descends to an action of $G_{\ad}$ on $G$, for which we use the same notation. For every $h\in G$ we define the \emph{adjoint centralizer}
\[Z_{\ad}(h)\colonequals \{a\in G_{\ad}\mid aha^{-1}=h\}.\]
Note that $Z_{\ad}(h)=Z_{G}(h)/Z_{G}$, where $Z_{G}$ is the center of $G$, and we have the following simple lemma.

\begin{lemma}
\label{connex}
Suppose that $h\in G$ is a regular element. Then $Z_{\ad}(h)$ is connected.
\end{lemma}
\begin{proof}
Let $h=us$ be the Jordan decomposition of $h$ into a unipotent part $u$ and a semisimple part $s$. Let $L=Z_{G}(s)$ be the centralizer of $s$ in $G$. Because $G$ is simply-connected, the reductive group $L$ is connected. 

Since $h$ is regular, the unipotent element $u$ is regular in $L$ and by \cite[Lemma 4.3]{spr:66} we have
\[Z_{G}(h)=Z_{L}(u)=Z_{L}\times Z_{U_L}(u).\]
Here $U_L$ is the unique maximal unipotent subgroup of $L$ which contains $u$, and the second factor $Z_{U_L}(u)$ is connected by \cite[Theorem 4.11]{spr:66}. 

Write $L_{\ad}\colonequals L/Z_{G}\subset G_{\ad}$ for the Levi subgroup of $G_{\ad}$ which is the image of $L$. Let $R$ be the weight lattice of $G_{\ad}$, and let $I$ be the simple roots which do not vanish on the center $Z_{L_{\ad}}$. Then the number of connected components of $Z_{L_{\ad}}$ is given by the torsion subgroup of the abelian group
\[R/\mathbb{Z}I.\]
Since $G_{\ad}$ is of adjoint type, $R$ is equal to the root lattice and this torsion subgroup is trivial. Therefore $Z_{L_{\ad}}$ is connected.

The center $Z_{G}$ is the kernel of the group homomorphism 
\begin{equation}
\label{groups}
Z_{L}\longrightarrow Z_{L_{\ad}}.
\end{equation}
Fix an element $\bar{g}\in Z_{L_{\ad}}$ and let $g\in L$ be a point in the preimage of $\bar{g}$. Then the set
\[S_g=\{glg^{-1}l^{-1}\mid l\in L\}\]
is contained in $Z_G$. Since $L$ is connected, the set $S_g$ is connected. Moreover, since $G$ is a semisimple group, the center $Z_G$ is discrete and therefore $S_g$ consists only of the identity element. This implies that $g\in Z_L$, so the group homomorphism \eqref{groups} is surjective. Therefore
\[Z_{\ad}(h)=Z_{G}(h)/Z_{G}\cong Z_{L_{\ad}}\times Z_{U_L}(u)\]
and, since $Z_{L_{\ad}}$ is connected, it follows that $Z_{\ad}(h)$ is also connected.
\end{proof}

\begin{definition}
The (\emph{multiplicative}) \emph{universal centralizer} associated to $G$ is the affine variety
\[\Z\colonequals \left\{(a,h)\in G_{\ad}\times \Sigma \mid a\in Z_{\ad}(h)\right\}.\]
\end{definition}

\begin{remark}
Consider the space of commuting pairs
\[\mathcal{C}\colonequals \{(a,h)\in G_{\ad}\times G^\text{r}\mid a\in Z_{\ad}(h)\}\]
in which the second element is regular. The group $G$ acts on $\mathcal{C}$ diagonally, and via \eqref{phase} there is an isomorphism
\[\Z\cong\mathcal{C}/G.\]
Here the right-hand side is the categorical quotient of $\mathcal{C}$ by the $G$-action, which is studied in \cite{bez.fin.mir:05}. In Proposition 2.8 and Theorem 2.15 of \emph{loc. cit.} it is shown that, when $G$ is simply-laced, its coordinate ring is isomorphic to the equivariant $K$-theory of the affine Grassmannian of the Langlands dual group $G^\vee$. 
\end{remark}

We will consider the \emph{double}
\[\DG\colonequals G_{\ad}\times G,\]
which is the quotient of the space $D(G)$ in Example \ref{double} by the action of the finite center $Z_{G}$ on the left. The $G\times G$-action \eqref{paramact}, the bivector $\pi$ \eqref{firstbiv}, and the moment map $\mu$ \eqref{moment} all descend to $\DG$. Keeping this notation, $(\DG, \pi,\mu)$ is a nondegenerate quasi-Poisson $G\times G$-variety. 

\begin{remark}
\label{broids}
We may view $\DG$ as a constant algebraic group scheme over $G_{\ad}$. On the other hand, letting $\mathfrak{g}$ be the Lie algebra of $G_{\ad}$ and using the Killing form to identify $\mathfrak{g}^*\cong\mathfrak{g}$, the cotangent bundle 
\[T^*G_{\ad}\cong G_{\ad}\times\mathfrak{g}\]
becomes a bundle of Lie algebras over $G_{\ad}$. The double 
\[\DG=G_{\ad}\times G\]
is then its simply-connected integration. 
\end{remark}

In view of Example \ref{ex:double}, the multiplicative universal centralizer
\[\Z=\mu^{-1}(\Sigma\times\iota(\Sigma))=\{(a,h)\in G_{\ad}\times\Sigma\mid aha^{-1}=h\}\]
sits inside $\DG$ as a symplectic Steinberg slice. In particular, as in \cite{fin.tsy:19}, through isomorphism \eqref{phase} $\Z$ is equipped with an integrable system given by the invariant generators of $\mathbb{C}[\Ts]^{W}$.

%
%
%
%
%
%
%
\subsection{The wonderful compactification}
Let $l$ be the rank of $G$. The \emph{wonderful compactification} $\Gbar$ is a canonical, smooth, $G\times G$-equivariant compactification of $G_{\ad}$ which was introduced by de Concini and Procesi \cite{dec.pro:83}. We recall some of its structure theory, following \cite{eve.jon:08}.  It is a smooth projective variety which contains $G_{\ad}$ as an open dense subset and on which $G$ acts by extensions of the left- and right-multiplication. The boundary
\[D\colonequals \Gbar\backslash G_{\ad}\]
is a simple normal crossing divisor with irreducible components $D_1,\ldots, D_l$. 

The $G\times G$ orbits on $\Gbar$ are in bijection with subsets of the simple roots in the sense that, for any $I\subset\{1,\ldots,l\}$, the closure of the orbit $\mathcal{O}_I$ is the corresponding partial intersection of divisor components
\[\overline{\mathcal{O}_I}=\bigcap_{i\not\in I}D_i.\]
In particular, the closure of each orbit is smooth. 

The subset $I\subset\{1,\ldots,l\}$ determines a ``positive'' parabolic subgroup $P_I$, generated by the ``positive'' Borel $B$ and the simple root spaces indexed by $I$. Write $P_I^-$ for the opposite parabolic and $L_I$ for their common Levi component. Let $U^\pm_I\subset P^\pm_I$ be the unipotent radicals, and denote by $\mathfrak{p}^\pm_I$, $\mathfrak{u}_I^\pm$, and $\mathfrak{l}_I$ the Lie algebras of these subgroups. Each orbit $\mathcal{O}_I$ has a distinguished basepoint
\[z_I\in\mathcal{O}_I\]
whose $G\times G$-stabilizer is
\begin{equation}
\label{evestab}
\Stab_{G\times G}(z_I)\colonequals \left\{(us,vt)\in P_I\times P_I^-\mid u\in U_I, v\in U_I^-, s,t\in L_I, st^{-1}\in Z_{L_I}\right\}.
\end{equation}
It follows that $\mathcal{O}_I$ is a fiber bundle over the product of partial flag varieties $G/P_I\times G/P_I^-$, with fiber isomorphic to the adjoint group $L_I/Z_{L_I}$. This extends to a smooth fibration
\begin{equation*}
\begin{tikzcd}[row sep=large]
\overline{L_I/Z_{L_I}}\arrow[hook, r]&\overline{\mathcal{O}_I}\arrow[d] \\
									&G/P_I\times G/P_I^-
\end{tikzcd}
\end{equation*}
whose fiber is the wonderful compactification of $L_I/Z_{L_I}$.

The wonderful compactification $\Gbar$ is \emph{log-homogeneous} in the sense of \cite{bri:09}---that is, the logarithmic infinitesimal action map
\[\text{act}_D\colon \Gbar\times\mathfrak{g}\times\mathfrak{g}\longrightarrow T_D\Gbar\]
is surjective. Let $\kappa$ be the Killing form on $\mathfrak{g}$. In the short exact sequence of vector bundles over $\Gbar$
\[ 0\longrightarrow\ker(\text{act}_D)\longrightarrow\Gbar\times\mathfrak{g}\times\mathfrak{g}\longrightarrow T_D\Gbar\longrightarrow0,\]
the kernel $\ker(\text{act}_D)$ is Lagrangian relative to the form $(\kappa,-\kappa)$ on the direct sum $\mathfrak{g}\times\mathfrak{g}$ \cite[Example 2.5]{bri:09}. It follows that 
\begin{equation}
\label{embedding}
\ker(\text{act}_D)\cong\Tlog.
\end{equation}
This identifies the log-cotangent bundle $\Tlog$ with a subbundle of the trivial bundle $\Gbar\times\mathfrak{g}\times\mathfrak{g}$, extending the embedding
\begin{align*}
T^*G_{\ad}\cong G_{\ad}\times\mathfrak{g} &\hooklongrightarrow G_{\ad}\times\mathfrak{g}\times\mathfrak{g} \\
	(a,x)		&\longrightarrow ({a, \Ad_a x, x}).
	\end{align*}
Under \eqref{embedding}, the fiber of the log-cotangent bundle at the orbit basepoint $z_I\in\mathcal{O}_I$ is
\[T^*_{D,z_I}\Gbar\cong\mathfrak{p}_I\times_{\mathfrak{l}_I}\mathfrak{p}^-_I.\]

\begin{remark}
Via \eqref{embedding}, the log-cotangent bundle $\Tlog$ is a bundle of Lie algebras over $\Gbar$. In analogy with Remark \ref{broids}, we will show in the next section that it integrates to a smooth subgroup scheme of the constant group scheme 
\[\Gbar\times G\times G\longrightarrow\Gbar.\]\\
\end{remark}

%
%
%
%
%
%
%
\section{The logarithmic double}
\label{sec:dlog}
In this section we recall the Vinberg monoid, and we use it to construct an enlargement of the double $\DG$ to a group scheme $\DGlog$ over the wonderful compactification $\Gbar$. The nondegenerate quasi-Poisson structure on $\DG$ will extend to a log-nondegenerate quasi-Poisson structure on $\DGlog$.

\subsection{Construction of $\DGlog$}
The \emph{Vinberg monoid} $V_G$, introduced in \cite{vin:95}, is a normal affine algebraic semigroup whose locus of invertible elements is the \emph{enhanced group}
\[\Genh\colonequals G\times_{Z_{G}} T.\]
There are natural projections
\begin{equation}
\label{twomaps}
\begin{tikzcd}
			&\Genh	\arrow{ld}\arrow{rd}		&\\
G_{\ad}			&											&T
\end{tikzcd}
\end{equation}
---the first is a principal $\Ts$-bundle, and the second is the abelianization of the group $\Genh.$ The second map extends to
\begin{equation*}
\alpha\colon V_G	\longrightarrow\Tbar,
\end{equation*}
where 
\[\Tbar=\spec\mathbb{C}[t^{\alpha_1},\ldots,t^{\alpha_l}]\cong\mathbb{C}^l.\]
Here $\alpha_1,\ldots,\alpha_l$ are the simple roots and $t^\lambda\in\mathbb{C}[T]$ is the function on $T$ given by the weight $\lambda$. The space $\Tbar$ is an abelian monoid into which the adjoint torus embeds as the group of units via the map
\[t	\longmapsto 			(\alpha_1(t),\ldots,\alpha_l(t)).\]
The morphism $\alpha$ is the abelianization of the monoid $V_G$.

The monoid $V_G$ carries an action of $G\times G\times\Ts$ that extends the natural action on the enhanced group, and $\alpha$ is $G\times G$-invariant. In particular, every fiber of $\alpha$ contains an open dense $G\times G$-orbit. The \emph{nondegenerate locus} $\Vino\subset V_G$ is the quasi-affine open dense subvariety whose intersection with each fiber of $\alpha$ is this maximal orbit. We obtain a diagram
\begin{equation*}
\begin{tikzcd}
			&\Vino	\arrow{ld}[swap]{\tauo}\arrow{rd}{\alphao}		&\\
\Gbar		&									&\Tbar,
\end{tikzcd}
\end{equation*}
whose pullback along the inclusion $\Genh\hooklongrightarrow \Vino$ is \eqref{twomaps}.

Now $\tauo$ and $\alphao$ are smooth morphisms, $\tauo$ is a principal $\Ts$-bundle, and the $G\times G$-stabilizer of any point $v\in\tauo^{-1}(z_I)$ is 
\begin{equation}
\label{eq:stabs}
\Stab_{G\times G}(v)=P_I\times_{L_I}P_I^-.
\end{equation}

Let $G\times G$ act on $\Gbar\times G\times G$ via 
\[(g,h)\cdot(a,x,y)=(gah^{-1}, gxg^{-1}, hyh^{-1})\]
for $(g,h)\in G\times G$ and $(a,x,y)\in\Gbar\times G\times G.$

\begin{proposition}
\label{logdouble}
There is a smooth, closed, $G\times G$-stable subgroup scheme $\DGlog\subset\Gbar\times G\times G$ whose fiber over the basepoint $z_I\in\Gbar$ is
\[P_I\times_{L_I}P_I^-.\]
\end{proposition}
\begin{proof}
Since $\alphao$ is smooth, the fiber product $\Vino\times_{\Tbar}\Vino$ is a smooth variety. The action morphism
\begin{align}
\label{vinmap}
\Vino\times G\times G 	&\longrightarrow 	\Vino\times_{\Tbar}\Vino \\
		(v,g,h) 		&\longrightarrow 	(v,gvh^{-1})\nonumber
		\end{align}
is smooth and surjective, because every fiber of $\alphao$ is a single $G\times G$-orbit. The preimage of the diagonal 
\[\Vino\hooklongrightarrow  \Vino\times_{\Tbar}\Vino\]
under \eqref{vinmap} is the smooth family of stabilizers
\[\mathcal{S}=\left\{(v,g,h)\in \Vino\times G\times G \mid (g,h)\in\Stab_{G\times G}(v)\right\},\]
defined for example in \cite[Appendix D]{dri.gai:15}. 

Because the action of $G\times G$ commutes with the action of $ T$, for any $v\in\Vino$ and $t\in T$ we have 
\[\Stab_{G\times G}(v)=\Stab_{G\times G}(t\cdot v).\]
Therefore the group scheme of stabilizers $\mathcal{S}$ descends through the principal $\Ts$-bundle $\tauo$ to a smooth, closed, $G\times G$-stable subvariety
\begin{equation*}
\DGlog\subset \Gbar\times G\times G.
\end{equation*}
By \eqref{eq:stabs}, the fiber of $\DGlog$ over $z_I\in\Gbar$ is $P_I\times_{L_I}P_I^-.$
\end{proof}

The group scheme $\DGlog$, which we call the \emph{logarithmic double}, integrates the bundle of Lie algebras given by the log-cotangent bundle
\[\Tlog\subset\Gbar\times\mathfrak{g}\times\mathfrak{g}.\]
Its fiber at the identity element $1\in G_{\ad}$ is the diagonal subgroup
\[\{(g,g)\mid g\in G\}\subset G\times G.\]
Since $\DGlog$ is $G\times G$ stable, it follows that its fiber at any point $a\in G_{\ad}$ is
\[\{(aga^{-1},g)\mid g\in G\}.\]
Therefore the logarithmic double $\DGlog$ is the closure of the image of the embedding
\begin{align}
\label{bottom}
\DG &\hooklongrightarrow \Gbar\times G\times G \\
(a,g) &\longmapsto (a, aga^{-1}, g). \nonumber
\end{align}
The diagram
\begin{equation*}
\begin{tikzcd}
\DG			\arrow[hookrightarrow]{r} \arrow{d}	&\DGlog	\arrow{d}\\
G_{\ad}				\arrow[hookrightarrow]{r}		&\Gbar,
\end{tikzcd}
\end{equation*}
is Cartesian, and $\DG$ is exactly the restriction of $\DGlog$ to the open dense copy of $G_{\ad}$ which sits inside $\Gbar$.
\vspace{.1in}

%
%
%
%
%
%
%
\subsection{The quasi-Poisson structure on $\DGlog$}
In view of the previous section, the nondegenerate quasi-Poisson variety $(\DG,\pi,\mu)$ sits inside the logarithmic double $\DGlog$ as an open dense subset. Its complement is a simple normal crossing divisor, and for simplicity we abuse notation to denote it by $D$. We will show that the quasi-Poisson bivector $\pi$ extends to a logarithmic bivector on $\DGlog$, and that this gives $\DGlog$ the structure of a log-nondegenerate quasi-Poisson manifold in the sense of Section \ref{sec:lognon}.

Using the notation of Section \ref{sec:background} and summing over repeated indices, define a bivector on the space $\Gbar\times G\times G$ by
\begin{equation}
\label{finbiv}
\pib=\frac{1}{2}\left(e_i^{1L}\wedge (e_i^{3L}+e_i^{3R})
	+e_i^{1R}\wedge (e_i^{2L}+e_i^{2R})+ e_i^{2R}\wedge e_i^{2L}+ e_i^{3L}\wedge e_i^{3R}\right).
	\end{equation}
Let the morphism $\mubar$, which extends the moment map $\mu\colon \DG\longrightarrow G\times G$ first defined in \eqref{moment}, be the composition
\begin{equation}
\label{barmoment}
\begin{tikzcd}[row sep=large]
 \DGlog\arrow[hook]{r}\arrow[rd, swap, "\mubar"]		& \Gbar\times G\times G \arrow{d} \\
									&G\times G,
\end{tikzcd}
\end{equation}
where the vertical arrow is
\begin{align*}
\Gbar\times G\times G	&\longrightarrow	G\times G \\
		(a,g,h)		&\longmapsto 		(g,h^{-1}).\qedhere
\end{align*}

\begin{proposition} 
\label{dgqpoisson}
The bivector $\pib$ is tangent to $\DGlog$, and $(\DGlog,\pib,\mubar)$ is a quasi-Poisson variety whose unique open dense nondegenerate leaf is $(\DG, \pi, \mu)$.
\end{proposition}
\begin{proof}
It is enough to show that the restriction of $\pib$ to 
\[\DG\subset\Gbar\times G\times G\]
agrees with $\pi$. This will imply that $\pib$ is tangent to $\DGlog$, which is the closure of $\DG$. Moreover, since $\pi$ satisfies the quasi-Poisson condition \eqref{qpcond} along $\DG$, $\pib$ will satisfy \eqref{qpcond} along $\DGlog$.

Recall that the embedding of $\DG$ into $\Gbar\times G\times G$ fits into the commutative diagram
\begin{equation*}
\begin{tikzcd}
D(G)			\arrow[hookrightarrow]{r} \arrow{d}	&G\times G\times G	\arrow{d}\\
\DG				\arrow[hookrightarrow]{r}			&\Gbar\times G\times G,
\end{tikzcd}
\end{equation*}
where $D(G)$ is as defined in Example \ref{double}. The top horizontal arrow is
\begin{align*}
D(G)=G\times G&\hooklongrightarrow G\times G\times G \\
		(g,h)&\longmapsto (g, gh, hg).
		\end{align*}
The bottom horizontal arrow is \eqref{bottom}, and the vertical arrows are quotients by the left action of the center $Z_G$. Therefore, from Example \ref{double}, it is sufficient to check that the pushforward of 
\[\frac{1}{2}\left(e_i^{1L}\wedge e_i^{2R}+e_i^{1R}\wedge e_i^{2L}\right)\in\Gamma\left(\wedge^2TD(G)\right)\]
along the top arrow of this diagram agrees with \eqref{finbiv}.

At the point $(g, gh, hg)$ the vectors which constitute $\pi$ push forward to
\begin{align*}
&e_i^{1L}\longmapsto e_i^{1L}+ (\Ad_g e_i)^{2R}+ e_i^{3L} \\
&e_i^{1R}\longmapsto e_i^{1R}+ e_i^{2R}+(\Ad_{g^{-1}}e_i)^{3L}  \\
&e_i^{2L}\longmapsto e_i^{2L} + (\Ad_{g^{-1}}e_i)^{3L} \\
&e_i^{2R}\longmapsto(\Ad_ge_i)^{2R} +  e_i^{3R}.
\end{align*}
Therefore, at $(g, gh, hg)$ the bivector $\pi$ is half the expression
\begin{align}
\label{bigbiv}
e_i^{1L}\wedge (\Ad_g&e_i)^{2R}+ e_i^{1L}\wedge e_i^{3R}+ e_i^{3L}\wedge e_i^{3R}  \\
	&+ (\Ad_g e_i)^{2R}\wedge e_i^{3R}+ e_i^{3L} \wedge (\Ad_ge_i)^{2R} \nonumber\\
	&+e_i^{1R}\wedge e_i^{2L}+e_i^{1R}\wedge (\Ad_{g^{-1}}e_i)^{3L}+ e_i^{2R}\wedge e_i^{2L} \nonumber\\
	&+ e_i^{2R}\wedge (\Ad_{g^{-1}}e_i)^{3L}+ (\Ad_{g^{-1}}e_i)^{3L}\wedge e_i^{2L}.\nonumber
\end{align}

Since $\Ad_g$ and $\Ad_{g^{-1}}$ are orthogonal operators relative to the Killing form, and since we are summing over repeated indices, the first terms in the first and third lines simplify: 
\begin{align*}
&e_i^{1L}\wedge (\Ad_ge_i)^{2R}=(\Ad_ge_i)^{1R}\wedge (\Ad_ge_i)^{2R}=e_i^{1R}\wedge e_i^{2R};\\
&e_i^{1R}\wedge (\Ad_{g^{-1}}e_i)^{3L}=(\Ad_{g^{-1}}e_i)^{1L}\wedge (\Ad_{g^{-1}}e_i)^{3L}=e_i^{1L}\wedge e_i^{3L}.
\end{align*}
Moreover, applying orthogonality again, the terms in the last row become
\[e_i^{2R}\wedge (\Ad_{g^{-1}}e_i)^{3L} =(\Ad_ge_i)^{2R}\wedge e_i^{3L}\]
and
\[(\Ad_{g^{-1}}e_i)^{3L}\wedge e_i^{2L}=e_i^{3L}\wedge (\Ad_ge_i)^{2L}=e_i^{3R}\wedge (\Ad_ge_i)^{2R}.\]
Therefore the second and fourth lines of \eqref{bigbiv} sum to zero, and we see that \eqref{bigbiv} agrees exactly with \eqref{finbiv}.
\end{proof}

\begin{proposition}
\label{prop:lognon}
The quasi-Poisson variety $(\DGlog,\pib,\mubar)$ is log-nondegenerate.
\end{proposition}
\begin{proof}
It is clear from \eqref{finbiv} that $\pib$ is a logarithmic bivector, because the action of $G\times G$ on $\Gbar$ preserves the boundary divisor. We will check that $\pib$ satisfies condition \eqref{lognon}---that the morphism of vector bundles
\[\pib^\#_D\oplus\rho_D\colon T^*_D\DGlog\oplus\mathfrak{g}\oplus\mathfrak{g}\longrightarrow T_D\DGlog\]
is surjective. By $G\times G$-equivariance, is sufficient to check this at a point of the form $(z_I,x,y)\in \DGlog$. We begin by making a fixed choice of orthonormal basis.

Let $R_0$ be the set of weights of the $T$-action on $\mathfrak{g}$, with multiplicity and including $0$. Write $R^+$ for the subset of $R_0$ consisting of positive roots. Choose a basis of generalized eigenvectors
\[\mathcal{B}\colonequals \{E_\alpha\mid\alpha\in R_0\}\subset\mathfrak{g}.\]
By scaling $E_\alpha$ if necessary, we obtain an orthonormal basis
\[\{E_\alpha\mid \alpha=0\}\cup\{E_{\alpha}\pm E_{-\alpha}\mid \alpha\in R^+\}\]
of $\mathfrak{g}$ relative to the Killing form. In this basis, the bivector $\pib$ from \eqref{finbiv} becomes
\[\pib=E_\alpha^{1L}\wedge(E_\alpha^{3L}+E_\alpha^{3R})+E_\alpha^{1R}\wedge(E_\alpha^{2L}+E_\alpha^{2R})+E_\alpha^{2R}\wedge E_\alpha^{2L}+E_\alpha^{3L}\wedge E_\alpha^{3R},\]
where once again we sum over the repeated index $\alpha\in R_0$.

As in \eqref{embedding}, the infinitesimal action map
\[\mathfrak{g}\times\mathfrak{g}\longrightarrow T_{D,z_I}\Gbar\]
is surjective with kernel $\mathfrak{p}_I\times_{\mathfrak{l}_I}\mathfrak{p}_I^-$. In particular, the vector fields $\{E_\alpha^L, E_\alpha^R\mid\alpha\in R_0\}$ span the log-cotangent space of $\Gbar$ at every point. Therefore the image  of $\rho_D$ at $(z_I,x,y)$, which is spanned by the logarithmic vectors
\[\left\{E_\alpha^{1L}+E_\alpha^{3L}-E_\alpha^{3R},\, E_\alpha^{1R}+E_\alpha^{2L}-E_\alpha^{2R}\right\},\]
contains a subspace of dimension $\dim G$ which is not parallel to the fiber.

Let $\{\theta_\alpha\mid\alpha\in R_0\}$ be the basis of $\mathfrak{g}^*$ dual to $\mathcal{B}$. Since the logarithmic vector fields 
\[\left\{E_\alpha^{1L}\mid E_\alpha\in\mathfrak{p}_I^-\right\}\subset\Gamma(T^*_D\DGlog)\]
are linearly independent at $(z_I,x,y)\in\DGlog$, the corresponding $1$-forms 
\[\left\{\theta_\alpha^{1L}\mid E_\alpha\in\mathfrak{p}_I^-\right\}\subset \Gamma(T^*\DG)\]
extend to logarithmic $1$-forms in a neighborhood of $(z_I,x,y)\in\DGlog$. By the same argument, the same is true for 
\[\left\{\theta_\alpha^{1R}\mid E_\alpha\in\mathfrak{p}_I\right\}\subset \Gamma(T^*\DG).\]

Applying $\pib_D^\#$ to these logarithmic $1$-forms at $(z_I,x,y)\in\DGlog$, we obtain
\[\pib_D^\#(\theta_\alpha^{1L})=\begin{cases}
							E_\alpha^{3L}+E_\alpha^{3R},\quad\qquad\qquad\qquad\text{ if }E_\alpha\in\mathfrak{p}_I^-\backslash\mathfrak{l}_I \\
							E_\alpha^{2L}+E_\alpha^{2R}+E_\alpha^{3L}+E_\alpha^{3R}, \quad\text{ if }E_\alpha\in\mathfrak{l}_I
							\end{cases}\]
and
\[\pib_D^\#(\theta_\alpha^{1R})=\begin{cases}
							E_\alpha^{2L}+E_\alpha^{2R},\quad\qquad\qquad\qquad\text{ if }E_\alpha\in\mathfrak{p}_I\backslash\mathfrak{l}_I \\
							E_\alpha^{2L}+E_\alpha^{2R}+E_\alpha^{3L}+E_\alpha^{3R}, \quad\text{ if }E_\alpha\in\mathfrak{l}_I. 
							\end{cases}\]
This implies that the image of $\pib^\#_D$ contains a subspace of dimension $\dim G$ which is parallel to the fiber. It follows that, at the point $(z_I,x,y)$,
\[\dim\left(\im(\pib^\#_D\oplus\rho_D)\right)=2\dim G.\]
Therefore this morphism of vector bundles is surjective.\qedhere\\
\end{proof}

%
%
%
%
%
%
%
\section{The relative compactification of $\Z$}
\label{sec:part}
Consider the relative compactification
\begin{equation*}
\Zbar=\left\{(a,h)\in\Gbar\times\Sigma\mid a\in\overline{Z_{\ad}(h)}\right\}.
\end{equation*}
By realizing $\Zbar$ as a Steinberg slice in $\DGlog$, we will use the results of the previous sections to show that it is a smooth algebraic variety whose boundary is a simple normal crossing divisor, and that the symplectic structure on $\Z$ defined (up to a finite central quotient) in Example \ref{ex:double} extends to a log-symplectic structure on $\Zbar$. We will then describe the symplectic leaves of this structure.

\subsection{Construction of $\Zbar$}
We begin by characterizing the image and fibers of the compactified moment map $\mubar.$ In Section \ref{sec:stein} we defined the quotient map $\Xi\colon G\longrightarrow T/W$, whose fibers are the closures of the regular conjugacy classes. In view of diagram \eqref{barmoment}, the map $\mubar$ is proper, and we have the following description of its image.

\begin{lemma}
\label{prop:image}
The image of $\mubar$ is the closed subvariety
\[\immub\colonequals \left\{(g,h)\in G\times G\mid \Xi(g)=\Xi(h^{-1})\right\}\]
consisting of pairs of elements $(g,h)\in G\times G$ with the property that $g$ and $h^{-1}$ lie in the closure of the same conjugacy class.
\end{lemma}
\begin{proof}
Since $\mubar$ is proper, its image is closed, so it is the closure of the image of $\mu$. As in \eqref{image}, the image of $\mu$ is the collection of pairs 
\[\left\{(g,h)\in G\times G\mid g\text{ is conjugate to }h^{-1}\right\}.\]
The closure of this set is precisely $\immub$.
\end{proof}

\begin{lemma}
\label{normality}
The variety $\immub$ is normal. 
\end{lemma}
\begin{proof}
Because $\immub$ is the image of $\mubar$, it is irreducible of dimension 
\[2\dim G-l.\]
Let $f_1,\ldots,f_l\in\mathbb{C}[G]^{G}$ be a set of generators for the algebra of conjugation-invariant functions on $G$. Then
\[\immub=\left\{(g,h)\in G\times G\mid f_i(g)=f_i(h^{-1})\text{ for all }1\leq i\leq l\right\}.\]
In particular, $\Delta$ is the vanishing locus of exactly $l$ algebraically independent functions on $G\times G$. Therefore it is a complete intersection, and in particular it is Cohen--Macaulay.

The regular locus 
\[\immub^\text{r}=\left\{(g,h)\in \immub\mid g\text{ and }h\text{ are regular}\right\}\]
is a smooth open subset of $\immub$ because the differentials $df_1,\ldots,df_l$ are linearly independent at every point of $G^\text{r}$ \cite[Theorem 1.5]{ste:65}. Moreover, the complement of $\immub^\text{r}$ in $\Delta$ has codimension at least two \cite[Theorem 1.3]{ste:65}. Since $\immub$ is a Cohen--Macaulay variety with no singularities in codimension one, by Serre's criterion it is normal.
\end{proof}

\begin{lemma}
\label{connectedness}
The fibers of $\mubar$ are connected.
\end{lemma}
\begin{proof}
A general fiber of $\mubar$ is the closure in $\DGlog$ of a general fiber of $\mu$, which is connected by Lemma \ref{connex}. Since $\mubar$ is proper, by Stein factorization it decomposes as a composition
\[\mubar=f\circ g\]
where $f$ is a finite morphism and $g$ has connected fibers. 

Because the general fiber of $\mubar$ is connected, $f$ is a birational map. Moreover, since the image of $f$ is normal by Lemma \ref{normality}, it follows from Zariski's main theorem that all the fibers of $f$ are connected. Therefore the fibers of $\mubar$ are also connected.
\end{proof}

\begin{theorem}
\label{bigmain}
The variety $\Zbar$ is smooth and has a natural log-symplectic Poisson structure whose open dense symplectic leaf is $\Z$.
\end{theorem}
\begin{proof}
By Propositions \ref{dgqpoisson} and \ref{prop:lognon}, $\DGlog$ is a log-nondegenerate quasi-Poisson variety whose open dense leaf is the double $\DG$. There is a commutative diagram of moment maps
\begin{equation}
\label{comm}
\begin{tikzcd}[row sep=large]
\DG\arrow[hook]{r}\arrow[swap]{rd}{\mu}	&\DGlog\arrow{d}{\mubar} \\
							&G\times G.
							\end{tikzcd}
							\end{equation}

Two elements of $\Sigma$ are in the closure of the same conjugacy class if and only if they are equal. It follows from Lemma \ref{prop:image} that
\[\mubar^{-1}(\Sigma\times\iota(\Sigma))=\mubar^{-1}(\Sigma_\Delta).\]
Since $\Sigma\times\iota(\Sigma)$ is a Steinberg cross-section in $G\times G$, Theorem \ref{slices} implies that the preimage $\mubar^{-1}(\Sigma_\Delta)$ is a smooth subvariety of $\DGlog$ with a natural Poisson structure whose symplectic leaves are the intersections of $\mubar^{-1}(\Sigma_\Delta)$ with the nondegenerate leaves of $\DGlog$. This Poisson structure is log-symplectic by Proposition \ref{prop:stein-lognon}. It remains only to show that $\mubar^{-1}(\Sigma_\Delta)$ is isomorphic to $\Zbar$.

By Lemma \ref{connectedness}, the variety $\mubar^{-1}(\Sigma_\Delta)$ is connected. Since it is also smooth, it is irreducible, and therefore it is the closure in $\DGlog$ of $\mu^{-1}(\Sigma_\Delta)\subset\DG$. In particular, for any $h\in \Sigma$,
\[\mubar^{-1}(h,h^{-1})=\overline{\mu^{-1}(h,h^{-1})}\cong\overline{Z_{\ad}(h)}\subset\Gbar.\]
It follows that
\[\mubar^{-1}(\Sigma_\Delta)=\left\{(a,h,h^{-1})\in\Gbar\times G\times G\mid h\in\Sigma, a\in\overline{Z_{\ad}(h)}\right\}\cong\Zbar.\]
We obtain a commutative diagram
\begin{equation*}
\begin{tikzcd}[row sep=large]
\Z\arrow[hook]{r}\arrow[swap]{rd}	&\Zbar\arrow{d} \\
							&\Sigma,
							\end{tikzcd}
							\end{equation*}
which is the pullback of \eqref{comm} along the embedding $\Sigma\cong\Sigma_\Delta\hooklongrightarrow G\times G$. Since the horizontal arrow in this diagram is the restriction of a backward-Dirac map, it is a Poisson morphism. In particular, $\Z$ sits inside $\Zbar$ as the unique open dense symplectic leaf.
\end{proof}

\begin{remark}
\label{coulomb}
Suppose that $G$ is simply-laced and let $G^\vee$ be its Langlands dual group. Write $\mathcal{K}$ for the field $\mathbb{C}((t))$ of Laurent series, $\mathcal{O}$ for its ring of integers, and 
\[\gr\colonequals G^\vee(\mathcal{K})/G^\vee(\mathcal{O})\]
for the affine Grassmannian of $G^\vee$. In \cite[Proposition 2.8 and Theorem 2.15]{bez.fin.mir:05}, the authors prove that there is a natural isomorphism Poisson algebras
\begin{equation*}
\label{rings}
\mathbb{C}[\Z]\cong K^{G^\vee(\mathcal{O})}(\gr)
\end{equation*}
between the coordinate ring of the multiplicative universal centralizer and the equivariant $K$-theory of $\gr$. Here the right-hand side has a ring structure given by convolution, and its Poisson bracket comes from the one-parameter deformation $K^{G^\vee(\mathcal{O})\times\mathbb{C}^*}(\gr)$ induced by the loop rotation.

The algebra $K^{G^\vee(\mathcal{O})}(\gr)$ has a natural filtration indexed by the weight lattice $\Lambda$ of $G$, which is induced by the support in $G^\vee(\mathcal{O})$-orbit closures. Using the Vinberg monoid realization of the wonderful compactification, one easily shows (see \cite[Section 5]{bal:19}) that the relative compactification $\Zbar$ can also be obtained via relative Proj as
\[\Zbar\cong\text{Proj}_\Sigma\left(\text{Rees}_\Lambda (K^{G^\vee(\mathcal{O})}(\gr)\right).\]
Therefore, in the simply-laced case, our relative compactification $\Zbar$ is an example the Rees algebra approach to compactified Coulomb branches illustrated in \cite[Remark 3.7]{bra.fin.nak:19}.\\
\end{remark}

%
%
%
%
%
%
%
\subsection{Symplectic leaves}
By Proposition \ref{prop:stein-lognon}, the symplectic leaves of $\Zbar$ are the connected components of the intersections of $\Zbar$ with the nondegenerate leaves of $\DGlog$. Therefore we first describe the nondegenerate leaves of the quasi-Poisson variety $(\DGlog,\pib,\mubar)$. For this we need to analyze the image of the (non-logarithmic) bundle map
\[\pib^\#\oplus\rho\colon T^*\DGlog\oplus\mathfrak{g}\oplus\mathfrak{g}\longrightarrow T\DGlog.\]

Fix an index set $I\subset\{1,\ldots,l\}$, and write
\[\mathfrak{c}_I\colon P_I\longrightarrow P_I/[P_I,P_I]=:A_I\]
for the quotient of $P_I$ by its derived subgroup. The torus $A_I$ is the ``universal torus'' associated to the standard parabolic $P_I$. We first give a criterion for when two points in the fiber of $\DGlog$ above $z_I\in\Gbar$ are in the same nondegenerate leaf.

\begin{proposition}
\label{fiberleaf}
Let $(x,y),(x',y')\in P_I\times_{L_I}P_I^-$. Then $(z_I,x,y)$ and $(z_I, x',y')$ are in the same nondegenerate leaf of $(\DGlog,\pib,\mubar)$ if and only if 
\[\mathfrak{c}_I(x)=\mathfrak{c}_I(x').\]
\end{proposition}

\begin{remark}
The value of $\mathfrak{c}_I(x)$ depends only on the $L_I$-component of the element
\[x\in P_I=L_I\ltimes U_I.\]
Since points in $P_I\times_{L_I}P_I^-$ are pairs with the same Levi component, the proposition could instead be stated in an equivalent way relative to the second coordinate and the negative parabolic $P_I^-$.
\end{remark}

\begin{proof}
In order to determine the intersection of the fiber $\{z_I\}\times (P_I\times_{L_I}P_I^-)$ with each nondegenerate leaf, we will find which vectors in the image of $\pib^\#\oplus\rho$ are tangent to the fibers of $\DGlog$.

By \eqref{evestab}, the kernel of the infinitesimal action map
\[\mathfrak{g}\times\mathfrak{g}\longrightarrow T_{z_I}\Gbar\]
is the subalgebra of pairs
\[\{(u+s,v+t)\in\mathfrak{p}_I\times\mathfrak{p}_I^-\mid u\in\mathfrak{u}_I, v\in\mathfrak{u}_I^-, s,t\in\mathfrak{l}_I, s-t\in Z_{\mathfrak{l}_I}\}.\]
We use the same notation as in the proof of Proposition \ref{prop:lognon}. Viewed as a section of $\wedge^2T\DGlog$, at the point $(z_I,x,y)$ the value of the bivector $\pib$ is
\begin{align*}
\pib=\sum_{E_\alpha\in\mathfrak{p}_I^-\backslash Z_{\mathfrak{l}_I}}E_\alpha^{1L}\wedge(E_\alpha^{2L}+E_\alpha^{2R})
		&+\sum_{E_\alpha\in\mathfrak{p}_I\backslash Z_{\mathfrak{l}_I}}E_\alpha^{1R}\wedge(E_\alpha^{3L}+E_\alpha^{3R})\\
		&\qquad+\sum_{\alpha\in R_0}(E_\alpha^{2L}\wedge E_\alpha^{2R}+E_\alpha^{3R}\wedge E_\alpha^{3L}).
		\end{align*}
Therefore, the vectors in the image of $\pib^\#\oplus\rho$ which are parallel to the fiber of $\DGlog$ at $z_I\in\Gbar$ are given by the span of
\begin{align*}
\left\{E_\alpha^{2L}+E_\alpha^{2R}\mid E_\alpha\in\mathfrak{u}_I^-\right\}\cup&\left\{E_\alpha^{3L}+E_\alpha^{3R}\mid E_\alpha\in\mathfrak{u}_I\right\}\\
			&\qquad\quad\cup\left\{E_\alpha^{2L}+E_\alpha^{2R}+E_\alpha^{3L}+E_\alpha^{3R}\mid E_\alpha\in\mathfrak{l}_I\backslash Z_{\mathfrak{l}_I}\right\}.
\end{align*}
At each point this is the tangent space to the fibers of the smooth morphism
\begin{align*}
P_I\times_{L_I}P_I^-&\longrightarrow A_I\\
	(x,y)&\longrightarrow \mathfrak{c}_I(x).
	\end{align*}
Since these fibers are connected, it follows that two points $(z_I,x,y)$ and $(z_I, x',y')$ are in the same nondegenerate leaf if and only if they have the same image under this map. 
\end{proof}

Let $\DGlogi$ be the preimage of $\mathcal{O}_I\subset\Gbar$ under the structure map
\[\DGlog\longrightarrow \Gbar.\]
Since both $\pi$ and $\rho$ are tangent to the boundary of $\DGlog$, each orbit preimage $\DGlogi$
is a union of nondegenerate leaves. To extend the criterion of Proposition \ref{fiberleaf} to this preimage, we define the following data. 

For any $a\in\Gbar$, there exist group elements $g,h\in G$ such that $a=gz_Ih^{-1}$. We associate to this point a corresponding ``positive'' parabolic subgroup
\[P_a\colonequals gP_Ig^{-1},\]
which, in view of \eqref{evestab}, depends only on the point $a\in\Gbar$. There is a canonical identification of tori
\[P_a/[P_a,P_a]\cong P_I/[P_I,P_I]= A_I,\]
and we denote the corresponding quotient map by
\[\mathfrak{c}_a\colon P_a\longrightarrow P_a/[P_a,P_a]\cong A_I.\]

The preimage $\DGlogi$ is a locally trivial $G\times G$-equivariant fiber bundle over $\mathcal{O}_I$. In other words, there is an isomorphism
\begin{equation*}
\begin{tikzcd}[row sep=large]
\DGlogi\arrow[r, "\sim"]\arrow{rd}	&(G\times G)\times_{\Stab_{G\times G}(z_I)}(P_I\times_{L_I}P_I^-)\arrow{d} \\
							&\mathcal{O}_I,
							\end{tikzcd}
							\end{equation*}
Moreover, the map
\begin{align*}
(G\times G)\times_{\Stab_{G\times G}(z_I)}(P_I\times_{L_I}P_I^-)&\longrightarrow A_I \\
						[(g,h):(x,y)]&\longmapsto\mathfrak{c}_I(x)
						\end{align*}
is well-defined. Composing it with the isomorphism above, we get a smooth morphism
\begin{align*}
\DGlogi&\longrightarrow A_I \\
		(a,x,y)&\longmapsto \mathfrak{c}_a(x).\nonumber
		\end{align*}
In the following proposition we show that its fibers are precisely the nondegenerate quasi-Poisson leaves in $\DGlogi$.		

\begin{proposition}
\label{finalleaf}
Two points $(a,x,y), (b,w,z)\in\DGlogi$ are in the same nondegenerate leaf of $\pib$ if and only if 
\[\mathfrak{c}_a(x)=\mathfrak{c}_b(w).\]
\end{proposition}
\begin{proof}
There exist points 
\[(x',y'), (w',z')\in P_I\times_{L_I}P_I^-\]
such that $(a,x,y)$ is $G\times G$-conjugate to $(z_I,x',y')$ and $(b,w,z)$ is $G\times G$-conjugate to $(z_I,w',z')$. Since the nondegenerate leaves of $(\DGlog,\pib)$ are $G\times G$-stable, $(a,x,y)$ and $(b,w,z)$ are in the same leaf if and only if their translates $(z_I,x',y')$ and $(z_I,w',z')$ are in the same leaf. By Proposition \ref{fiberleaf}, this occurs if and only if 
\[\mathfrak{c}_I(x')=\mathfrak{c}_I(w').\]
But now $\mathfrak{c}_a(x)=\mathfrak{c}_I(x')$ and $\mathfrak{c}_b(w)=\mathfrak{c}_I(w')$, and the statement follows.
\end{proof}

The orbit stratification on $\Gbar$ induces a stratification 
\[\Zbar=\bigsqcup \Zbar_I\]
on $\Zbar$, where 
\[\Zbar_I\colonequals \Zbar\cap\DGlogi=\left\{(a,h)\in\Gbar\times\Sigma\mid a\in\overline{Z_{\ad}(h)}\cap\mathcal{O}_I\right\}.\]
By Theorem \ref{slices}(c), each stratum $\Zbar_I$ is a union of symplectic leaves, and Proposition \ref{finalleaf} has the following immediate corollary.

\begin{corollary}
The symplectic leaves of $\Zbar_I$ are the fibers of the smooth morphism
\begin{align*}
\Zbar_I&\longrightarrow A_I \\
(a,h)&\longmapsto \mathfrak{c}_a(h).\\
\end{align*}
\end{corollary}

%
%
%
%
%
%
%
\bibliographystyle{plain}
\bibliography{mult}

\end{document}